\documentclass[a4paper,10pt]{extarticle}

\newcommand{\thetitle}{How Many Eigenvalues of a Random \\ Symmetric Tensor are Real?}

\newcommand{\thelanguage}{english}

\usepackage{geometry}
\usepackage[format=plain, font=footnotesize]{caption}
\usepackage[utf8]{inputenc}
\usepackage{listings}
\usepackage{multirow}
\usepackage[\thelanguage]{babel}
\usepackage[babel]{csquotes}
\usepackage{amsthm}
\usepackage{amsmath,amssymb}
\usepackage[colorlinks=true, citecolor=blue]{hyperref}
\usepackage[\thelanguage]{cleveref}
\usepackage{mathdots}
\usepackage{color,calc,comment}
\usepackage{multirow}
\usepackage{enumitem}
\usepackage{tabularx}
\usepackage{graphicx}
\usepackage[all,2cell]{xy}
\usepackage{mathtools}

\newcommand{\HN}{\mathbb{N}}

\newcommand{\HQ}{\mathbb{Q}}
\newcommand{\HR}{\mathbb{R}}

\newcommand{\kC}{\mathfrak{C}}

\newcommand{\kO}{\mathfrak{O}}

\newcommand{\kR}{\mathfrak{R}}
\newcommand{\kS}{\mathfrak{S}}
\newcommand{\kT}{\mathfrak{T}}

\newcommand{\cA}{\mathcal{A}}

\newcommand{\cI}{\mathcal{I}}
\newcommand{\cJ}{\mathcal{J}}
\newcommand{\cK}{\mathcal{K}}
\newcommand{\cL}{\mathcal{L}}
\newcommand{\cM}{\mathcal{M}}

\newcommand{\cX}{\mathcal{X}}
\newcommand{\cY}{\mathcal{Y}}

\newcommand{\set}[1]{\left\{#1\right\}}
\newcommand{\cset}[2]{\left\{#1\mid #2\right\}}

\renewcommand{\d}{\mathrm{ d}}

\DeclareMathOperator*{\mean}{\mathbb{E}}

\numberwithin{equation}{section}
\numberwithin{figure}{section}

\theoremstyle{plain}
\newcounter{numbering} \numberwithin{numbering}{section}
\newtheorem{thm}[numbering]{Theorem}
\newtheorem{lemma}[numbering]{Lemma}
\newtheorem{prop}[numbering]{Proposition}
\newtheorem{cor}[numbering]{Corollary}

\theoremstyle{definition}
\newtheorem{dfn}[numbering]{Definition}
\theoremstyle{remark}

\newtheorem*{rem}{Remark}

\crefname{equation}{}{}
\crefname{equation}{}{}
\crefname{figure}{Figure}{Figures}
\crefname{section}{Section}{Sections}
\crefname{lemma}{Lemma}{Lemmata}
\crefname{prop}{Proposition}{Propositions}
\crefname{thm}{Theorem}{Theorems}
\crefname{cor}{Corollary}{Corollaries}
\crefname{dfn}{Definition}{Definitions}
\crefname{notation}{Notations}{Notations}
\crefname{rem}{Remark}{Remarks}
\crefname{claim}{Claim}{claims}
\crefname{table}{Table}{Tables}
\crefname{conj}{Conjecture}{Conjectures}
\begin{document}
\title{\thetitle}
\date{}
\author{Paul Breiding}
\maketitle
% \vspace{-2em}
\begin{abstract}
This article answers a question posed by Draisma and Horobet \cite{draisma-horobet}, who asked for a closed formula for the expected number of real eigenvalues of a random real symmetric tensor drawn from the Gaussian distribution relative to the Bombieri norm. This expected value is equal to the expected number of real critical points on the unit sphere of a Kostlan polynomial. We also derive an exact formula for the expected absolute value of the determinant of a matrix from the Gaussian Orthogonal Ensemble.
\end{abstract}
\thanks{ Max-Planck Institute for Mathematics in the Sciences, Leipzig, breiding@mis.mpg.de.\\ Partially supported by DFG research grant BU 1371/2-2.}
%%%%%%%%%%%%%%%%%%%%%%%%%%%%%%%%%%%%%%%%%%%%%%%%%%%%%%%%%%%%%%%%%%%%%%%
%%%%%%%%%%%%%%%%%%%%%%%%%%%%%%%%%%%%%%%%%%%%%%%%%%%%%%%%%%%%%%%%%%%%%%%
%%%%%%%%%%%%%%%%%%%%%%%%%%%%%%%%%%%%%%%%%%%%%%%%%%%%%%%%%%%%%%%%%%%%%%%
%%%%%%%%%%%%%%%%%%%%%%%%%%%%%%%%%%%%%%%%%%%%%%%%%%%%%%%%%%%%%%%%%%%%%%%
%%%%%%%%%%%%%%%%%%%%%%%%%%%%%%%%%%%%%%%%%%%%%%%%%%%%%%%%%%%%%%%%%%%%%%%
%%%%%%%%%%%%%%%%%%%%%%%%%%%%%%%%%%%%%%%%%%%%%%%%%%%%%%%%%%%%%%%%%%%%%%%
%%%%%%%%%%%%%%%%%%%%%%%%%%%%%%%%%%%%%%%%%%%%%%%%%%%%%%%%%%%%%%%%%%%%%%%
%%%%%%%%%%%%%%%%%%%%%%%%%%%%%%%%%%%%%%%%%%%%%%%%%%%%%%%%%%%%%%%%%%%%%%%
%%%%%%%%%%%%%%%%%%%%%%%%%%%%%%%%%%%%%%%%%%%%%%%%%%%%%%%%%%%%%%%%%%%%%%%
\section{Introduction}
The title of this article is a homage to the article \cite{real_eigenvalues}, in which Edelman, Kostlan and Shub compute the expected number of real eigenvalues of a matrix filled with i.i.d.\ standard Gaussian random variables. This model of a random matrix is extended to matrices of higher order, called \emph{tensors}.
In \cite{expected_Z_eigenvalues} the author computed the expected number of real eigenvalues of a random tensor, whose entries are i.i.d.\ standard Gaussian random variables. The content of this article is the computation of the expected number of real eigenvalues of a random  \emph{symmetric} tensor. Unlike symmetric matrices, symmetric tensors may have complex eigenvalues. But, as we will see, symmetric tensors have in a sense \emph{more} real eigenvalues than non-symmetric tensors.

In this article, a tensor $A=(a_{i_1,\ldots,i_p})_{1\leq i_1\leq n_1, \ldots,1\leq i_p\leq n_p}$ is an array of numbers arranged in~$p$ dimensions, where the $j$-th dimension is of length $n_j$. We call~$p$ the \emph{order} of $A$. We are interested in the case when $n:=n_1=\ldots=n_p$ are all equal. The space of real $n\times \cdots\times n$ tensors of order $p$ is denoted by $(\HR^{n})^{\otimes p}$.  For~$p\geq 3$ tensors are higher-dimensional analogues of matrices, which form the case $p=2$.
A tensor $A=(a_{i_1,\ldots,i_p})\in (\HR^{n})^{\otimes p}$ is called \emph{symmetric}, if for all permuations on $\pi\in\mathfrak{S}_p$ on~$p$ elements we have $a_{i_1,\ldots,i_p} = a_{i_{\pi(1)},\ldots,i_{\pi(p)}}$. We denote the vector space of symmetric tensors by $S^p(\mathbb{R}^n)$. As in \cite{expected_Z_eigenvalues} we say that a random real tensor $A\in\mathbb{R}^{n\times \cdots \times n}$ is a \emph{Gaussian tensor}, if it has density $(2\pi)^\frac{-n^p}{2}e^{-\frac{1}{2} \Vert A\Vert_F^2}$, where $\Vert A\Vert_F^2 := \sum_{1\leq i_1,\ldots,i_p\leq n} (a_{i_1,\ldots,i_p})^2$ is the square of the \emph{Frobenius norm}.

\begin{dfn}
A \emph{Gaussian symmetric tensor} $A\in S^p(\mathbb{R}^n)$ is a random tensor given by the density
$ k_{n,p}\,e^{-\frac{1}{2} \Vert A\Vert_F^2}$, where $k_{n,p}=(\int_{ S^p(\HR^n)}\,e^{-\frac{1}{2} \Vert A\Vert_F^2}\,\d A)^{-1}$ (Draisma and Horobet \cite{draisma-horobet} call the distribution given by this density the \emph{Gaussian distribution with respect to the Bombieri-norm}).
\end{dfn}

The complex number $\lambda \in \mathbb{C}$ is called
an \emph{eigenvalue} of the tensor $A\in \HR^{n\times \cdots\times n}$, if there exists a vector $v=(v_1,\ldots,v_n)\in \mathbb{C}^n$, such that the following equation holds:
\begin{equation}\label{def_eigenpairs}
Av^{p-1} := \bigg( \sum_{1\leq i_2,\ldots,i_p\leq n} \, a_{j,i_2,\ldots,i_p}\,v_{i_2}\cdots v_{i_p}\bigg)_{1\leq j\leq n} = \lambda v, \; \text{ and } v^Tv=1.
\end{equation}
The pair $(v,\lambda)$ is called an \emph{eigenpair} in this case.  For $p=2$, equation \cref{def_eigenpairs} is the defining equation of matrix eigenpairs.
The condition $v^Tv=1$ serves for selecting a point from each \emph{class} of eigenpairs: if $Av^{p-1} =\lambda v$, then also $A(tv)^{p-1} = (t^{p-2}\lambda)(t v)$. In particular, if $p$ is odd and if $\lambda$ is an eigenvalue of $A\in (R^n)^{\otimes p}$, then also $-\lambda$ is an eigenvlaue of $A$. To take into account this reflection property we make the following definition.
\begin{dfn}
If $p$ is odd we define the \emph{number of eigenvalues} of $A\in (\HR^n)^{\otimes p}$, to be the the number of solutions of \cref{def_eigenpairs} divided by two. For even $p$ even we define \emph{number of eigenvalues} of $A\in(\HR^n)^{\otimes p}$, to be the the number of solutions of \cref{def_eigenpairs}.
\end{dfn} For this definition Cartwright and Sturmfels \cite{sturmfels-cartwright} show that the number of complex eigenpairs for the generic tensor is $D(n,p):=\sum_{i=0}^{n-1}(p-1)^i$. In the following we use the notation
$$E(n,p):= \mean_{A \in S^p(\HR^n) \text{ Gaussian symmetric}} \text{ number of real eigenvalues of } A.$$
In \cref{cor} below we give an exact formula for $E(n,p)$. This complements our result from~\cite{expected_Z_eigenvalues}, where have given an exact formula for
$$E_\text{non-sym}(n,p):= \mean_{A \in (\HR^n)^{\otimes p} \text{ Gaussian}} \text{ number of real eigenvalues of } A.$$
This formula is given in terms of \emph{Gauss' hypergeometric function} $F(a,b,c,x)$ and the \emph{Gamma function} $\Gamma(n)$ (see \cref{def_hypergeom} and \cref{def_gamma} for their definitions):
$$E_\text{non-sym}(n,p) = \frac{2^{n-1}\,\sqrt{p-1}^{\,n}\,\Gamma(n-\tfrac{1}{2})}{\sqrt{\pi}\,p^{n-\frac{1}{2}}\,\Gamma(n)}\, \left(2(n-1) \, F(1,n-\tfrac{1}{2}, \tfrac{3}{2}, \tfrac{p-2}{p}) + F(1,n-\tfrac{1}{2}, \tfrac{n+1}{2}, \tfrac{1}{p}) \right).$$
Furthermore, we have shown the following asymptotic formulas:
$$\lim_{n\to \infty} \frac{E_\text{non-sym}(n,p)}{\sqrt{D(n,p)}} = \lim_{p\to \infty} \frac{E_\text{non-sym}(n,p)}{\sqrt{D(n,p)}} = 1.$$
Auffinger et.\ al.\ provide in \cite[Theorem 2.17]{AntonioAuffinger2013} the following formula:
  \begin{equation}\label{limit1}\lim_{n\to\infty} \frac{E(n,p)}{\sqrt{D(n,p)}} = \sqrt{\frac{8n}{\pi}} \,  (1+o(1)).
	\end{equation}
Comparing this with~\cref{limit1} it is fair to say that:
\begin{quote} \emph{For fixed $p$ and large $n$, and on the average, a real symmetric tensor has more eigenvalues than a real general tensor.}
\end{quote}
However, the point of view from Auffinger et.\ al.\ is not eigenvalues of tensors, but critical points of the polynomial $Ax^p := \sum_{1\leq i_1,i_2,\ldots,i_p\leq n} \, a_{i_1,i_2,\ldots,i_p}\,x_{i_2}\cdots x_{i_p}$ restricted to the unit sphere. If~$A\in S^p(\HR^n)$ is Gaussian symmetric, $Ax^p$ is called a \emph{Kostlan polynomial}. Indeed, every solution of \cref{def_eigenpairs}  corresponds to a critical point of $Ax^p$ on the unit sphere and vice versa. This point of view is also taken by Fyorodov, Lerario and Lundberg \cite{connected_comp}, who argue that each connected component of the zero set of a polynomial contains at least one critical point. Consequently, $E(n,d)$ yields information about the average topology of the zero set of a Kostlan polynomial.

Using the formula from \cref{cor} we get the following asymptotic formulas for large $p$:
	\begin{align}\label{asymptotic}
\lim_{p\to \infty} \frac{E(2m+1,p)}{\sqrt{D(2m+1,p)}} &=\frac{\sqrt{3\pi}}{\prod_{i=1}^{2m+1}\Gamma\left(\tfrac{i}{2}\right)}\sum_{1\leq i,j\leq m} \det(\Gamma_1^{i,j}) \, g_{i-1,j}, \;\text{ and }\\
\lim_{p\to \infty} \frac{E(2m,p)}{\sqrt{D(2m,p)}} &= \frac{\sqrt{3}}{\prod_{i=1}^{2m}\Gamma\left(\tfrac{i}{2}\right)}\sum_{0\leq i,j\leq m-1}  \det(\Gamma_2^{i,j})\,g_{i,j},\nonumber
	\end{align}
where $\Gamma_1^{i,j}$ and $\Gamma_2^{i,j}$ denote the following matrices depending on $m:=\lceil \tfrac{n}{2}\rceil$:
  \begin{equation}\label{gamma_matrices}
  \Gamma_1^{i,j}:=\left[\Gamma\left(r+s-\tfrac{1}{2}\right)\right]_{\substack{1\leq r\leq m,r\neq i\\1\leq s\leq m, s\neq j}} \quad \text{ and } \quad \Gamma_2^{i,j}=\left[\Gamma\left(r+s+\tfrac{1}{2}\right)\right]_{\substack{0\leq r\leq m-1, r\neq i\\0\leq s\leq m-1,s\neq j}}.
  \end{equation}
and where
\begin{equation}\label{def_g}
g_{i,j} = \begin{cases}
\displaystyle\frac{\sqrt{\pi}(2i+1)!}{(-1)^{i}\,3\, 2^{2i}\,i!}\, \;F\left(-i,\tfrac{1}{2},\tfrac{3}{2},\tfrac{-1}{3}\right) - \frac{   \Gamma\left(i+\frac{1}{2}\right)}{2\big(-\frac{3}{4}\big)^{i+1}}, &\text{ if } j = 0 \\[0.2cm]
\displaystyle\frac{ \,\Gamma\left(i+j+\frac{1}{2}\right)}{\tfrac{(1-2i-2j)}{(1-2i+2j)}\,\left(-\tfrac{3}{4}\right)^{i+j} } \;F\left(-2i,-2j+1,\tfrac{3}{2}-i-j,\tfrac{3}{4}\right), &\text{ if } j >0. \end{cases}
\end{equation}
Note that $g_{i,j}=\lim_{p\to\infty} g_{i,j}(p)$, where $g_{i,j}(p)$ are the rational polynomials from \cref{cor}.

At first glance the formulas in \cref{asymptotic} don't provide much insight, and unfortunately we don't know how to simplify them any further, nor do we know how to compute the leading order like in \cref{limit1}. But we have the following interesting corollary:
\begin{cor}\label{cor2} For fixed $n$ we have $\lim_{p\to \infty} \frac{E(n,p)}{\sqrt{D(n,p)}} \in \HQ(\sqrt{3}).$
\end{cor}
\begin{proof}
For all $k\in\HN$ we have $\Gamma(k+\tfrac{1}{2})=q\sqrt{\pi}$ for some $q\in\HQ$; see, e.g., \cite[43:4:3]{atlas}. In both formulas in \cref{asymptotic} there are as many $\sqrt{\pi}$s in the numerator as there are in the denominator.
\end{proof}

\begin{figure}[t]
\begin{center}
\includegraphics[width = 0.8\textwidth]{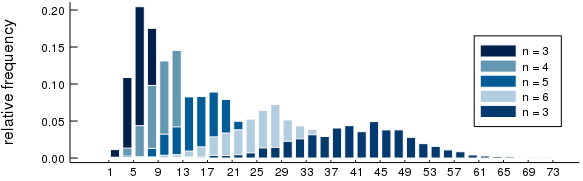}
\end{center}
\vspace{-1em}
\caption{The histogram shows the output of the following experiment: for $3\leq n\leq 7$ we sampled $10^4$ Gaussian symmetric tensors in $S^3(\HR^n)$ and computed the number of real eigenvalues using \cite{BT}. \cref{cor} predicts $E(3,3)\approx 5.28$,  $E(3,4)\approx 9.4$, $E(3,5)\approx15.75$,  $E(3,6)\approx 25.44$ and $E(3,7)\approx 40.1$.}
\label{fig1}
\end{figure}

Here is our main theorem. We give a proof in \cref{sec:proof2}.
\begin{thm}[An exact formula for $E(n,p)$]\label{cor}
We define the rational polynomial functions
\begin{equation*}
g_{i,j}(p) = \begin{cases}
\displaystyle\frac{\sqrt{\pi}(2i+1)!}{(-1)^{i} 2^{2i}\,i!}\,\frac{(p-2)^ip}{(p-1)^{i}(3p-2)} \;F\left(-i,\tfrac{1}{2},\tfrac{3}{2},\tfrac{-p^2}{(3p-2)(p-2)}\right) - \frac{   \Gamma\left(i+\frac{1}{2}\right)}{2\big(-\frac{3p-2}{4(p-1)}\big)^{i+1}}, &\text{ if } j = 0 \\[0.2cm]
\displaystyle\frac{ \,\Gamma\left(i+j+\frac{1}{2}\right)}{\tfrac{(1-2i-2j)}{(1-2i+2j)}\,\left(-\tfrac{3p-2}{4(p-1)}\right)^{i+j} } \;F\left(-2i,-2j+1,\tfrac{3}{2}-i-j,\tfrac{3p-2}{4(p-1)}\right), &\text{ if } j >0 \end{cases}
\end{equation*}
Then, for all $p\geq 3$ we have
\begin{align*}
E(2m+1,p) &=1+\frac{\sqrt{\pi}\,(p-1)^{m-1}\sqrt{(p-1)(3p-2)}}{\prod_{i=1}^{2m+1}\Gamma\left(\tfrac{i}{2}\right)}\sum_{1\leq i,j\leq m} \det(\Gamma_1^{i,j}) \, g_{i-1,j}(p), \text{ and }\\[0.2cm]
E(2m,p) &= \frac{(p-1)^{m-1}\,\sqrt{3p-2}}{\prod_{i=1}^{2m}\Gamma\left(\tfrac{i}{2}\right)}\sum_{0\leq i,j\leq m-1}  \det(\Gamma_2^{i,j})\,g_{i,j}(p).
\end{align*}
\end{thm}
Using essentially the same argument as for \cref{cor2} we can prove the following.
\begin{cor}\label{cor3} We have
 $E(2m+1,p)\in\HQ\big(\sqrt{(p-1)(3p-2)}\,\big)$ and $E(2m,p)\in\HQ\big(\sqrt{3p-2}\,\big).$
\end{cor}

\begin{table}
	\begin{center}
    \begin{tabular}{l|lc}
    \hline
    $n$ & $E(n,p)$ & $\lim\limits_{p\to\infty} \frac{E(n,p)}{\sqrt{D(n,p)}}$\\ \hline\\[-0.2cm]
 	2 & $\sqrt{3 \, p - 2}
$ &  $\sqrt{3}$\\[0.1cm]
 	3 & $1+\frac{4 \, {\left(p - 1\right)}^{\frac{3}{2}}}{\sqrt{3 \, p - 2}}$ & $\frac{4}{\sqrt{3}}$\\[0.1cm]
 	4 & $\frac{29 \, p^{3} - 63 \, p^{2} + 48 \, p - 12}{2 \, {\left(3 \, p -
2\right)}^{\frac{3}{2}}}$  & $\frac{29}{6\sqrt{3}}$\\[0.1cm]
 	5 & $1+\frac{2 \, {\left(5 \, p - 2\right)}^{2} {\left(p - 1\right)}^{\frac{5}{2}}}{{\left(3 \, p - 2\right)}^{\frac{5}{2}}}$ & $\frac{50}{9\sqrt{3}}$\\[0.1cm]
 	6 & $\frac{1339 \, p^{6} - 5946 \, p^{5} + 11175 \, p^{4} - 11240 \, p^{3} +
6360 \, p^{2} - 1920 \, p + 240}{8 \, {\left(3 \, p -
2\right)}^{\frac{7}{2}}}
  $ & $\frac{1339}{216\sqrt{3}}$\\[0.1cm]
 	7 & $1+\frac{{\left(1099 \, p^{4} - 2296 \, p^{3} + 2184 \, p^{2} - 992 \, p +
176\right)} {\left(p - 1\right)}^{\frac{7}{2}}}{2 \, {\left(3 \, p -
2\right)}^{\frac{9}{2}}}$ & $\frac{1099}{162\sqrt{3}}$\\[0.1cm]
    \hline
    \end{tabular}
		\end{center}
		\vspace{-1em}
		\caption{Formulas for $E(n,p)$ and $\lim_{p\to \infty}\tfrac{E(n,p)}{\sqrt{D(n,p)}}$ for $2\leq n \leq 7$. We used \textsc{sage} \cite{sage} to derive the middle column. The source code of the scripts are given in \cref{appendix}.\label{table0}}
\end{table}

%%%%%%%%%%%%%%%%%%%%%%%%%%%%%%%%%%%%%%%%%%%%%%%%%%%%%%%%%%%%%%%%%%%%%%%
%%%%%%%%%%%%%%%%%%%%%%%%%%%%%%%%%%%%%%%%%%%%%%%%%%%%%%%%%%%%%%%%%%%%%%%
%%%%%%%%%%%%%%%%%%%%%%%%%%%%%%%%%%%%%%%%%%%%%%%%%%%%%%%%%%%%%%%%%%%%%%%
%%%%%%%%%%%%%%%%%%%%%%%%%%%%%%%%%%%%%%%%%%%%%%%%%%%%%%%%%%%%%%%%%%%%%%%
\subsection{Random matrix theory}\label{sec:RMT}
Gaussian tensors of order $p=2$ are better known under another name: the \emph{Gaussian Orthogonal Ensemble}. If $A\in S^2(\HR^n)$ is a matrix from this ensemble, we write $A\sim \mathrm{GOE}(n)$. For $u\in \HR$, let us denote
		\begin{equation}\label{def_I_J}
		\cI_n(u):=\mean\limits_{A\sim \mathrm{GOE}(n)} \lvert \det(A-uI_n)\rvert,\quad\text{and}\quad \cJ_n(u):=\mean\limits_{A\sim \mathrm{GOE}(n)}\det(A-uI_n),
		\end{equation}
where $I_n$ is the $n\times n$ identity matrix. The proof of \cref{cor} is based on the computation of~$\cI_n(u)$.
We remark that $\lvert \cJ_n(u)\rvert \leq \cI_n(u)$ by the triangle inequality. A computation of $\cJ_n(u)$ can be found in \cite[Section 22]{mehta} and the ideas in this paper are inspired by the computations in this reference. The following result, which is new to our best knowledge, shows that $\cI_n(u)$ can be expressend in terms of~$\cJ_n(u)$ and a collection of \emph{Hermite polynomials}.

\begin{thm}[The expected absolute value of the determinant of a GOE matrix]\label{thm}
Let $u\in\HR$ be fixed. Define the functions $P_{-1}(x),P_0(x),P_1,(x),P_2(x),\ldots$ via
	\[P_k(x)=\begin{cases} -e^{\frac{x^2}{2}}\int_{t=-\infty}^x e^{-\tfrac{t^2}{2}}\,\d t,& \text{if } k=-1\\
	H_{e_k}(x),&\text{if } k=0,1,2,\ldots.\end{cases}\]
where $H_{e_k}(x)$ is the $k$-th (probabilist's) Hermite polynomial; see \cref{my_polynomials}.
Then, we have
\begin{align*}
	\cI_{2m}(u)&=\cJ_{2m}(u)+\frac{\sqrt{2\pi} \,e^{-\tfrac{u^2}{2}}}{\prod_{i=1}^{2m}\Gamma\left(\tfrac{i}{2}\right)}\,\sum_{1\leq i,j \leq m} \det(\Gamma_1^{i,j}) \; \det\begin{bmatrix} P_{2i-1}(u) & P_{2j}(u) \\ P_{2i-2}(u) & P_{2j-1}(u) \end{bmatrix}, \text{ and }\\
\cI_{2m-1}(u)&=\cJ_{2m-1}(u)  +  \frac{\sqrt{2}\,e^{-\tfrac{u^2}{2}}}{\prod_{i=1}^{2m-1}\Gamma\left(\tfrac{i}{2}\right)}\,\sum_{0\leq i,j\leq m-1}  \det(\Gamma_2^{i,j})\,  \det\begin{bmatrix} P_{2j}(u) &   P_{2i+1}(u)\\   P_{2j-1}(u) & P_{2i}(u)\end{bmatrix}.
\end{align*}
Here, $\Gamma_1^{i,j}$ and $\Gamma_2^{i,j}$ are the matrices from \cref{gamma_matrices}.
\end{thm}
\begin{rem}
The computation of $E(n,p)$ is based on the formula \cref{100} by Draisma and Horobet, which involves the expectation $\mean_{u\sim N(0,\sigma^2)}\cI_n(u)$ for $\sigma^2 = \tfrac{p}{2(p-1)}$.
In the recent article \cite{repeated_eigenvalues}, together with Khazhgali Kozhasov and Antonio Lerario, we have computed the volume of the set of matrices with repeated eigenvalues, and this computation is based on $\mean_{u\sim N(0,1/2)}\cI_n(u)^2$.
\end{rem}
%%%%%%%%%%%%%%%%%%%%%%%%%%%%%%%%%%%%%%%%%%%%%%%%%%%%%%%%%%%%%%%%%%%%%%%
\subsection{Organization of the article}
In the next section we give some preliminary material. Then, in \cref{sec:proof2} we prove \cref{cor}. In \cref{sec:integrals} we compute several integrals that are used in the proof of \cref{thm}, which we prove in \cref{sec:111}.
\subsection{Acknowledgements}
The author wants to thank Antonio Lerario for helpful remarks on the structure of this article.
%%%%%%%%%%%%%%%%%%%%%%%%%%%%%%%%%%%%%%%%%%%%%%%%%%%%%%%%%%%%%%%%%%%%%%%
%%%%%%%%%%%%%%%%%%%%%%%%%%%%%%%%%%%%%%%%%%%%%%%%%%%%%%%%%%%%%%%%%%%%%%%
%%%%%%%%%%%%%%%%%%%%%%%%%%%%%%%%%%%%%%%%%%%%%%%%%%%%%%%%%%%%%%%%%%%%%%%
%%%%%%%%%%%%%%%%%%%%%%%%%%%%%%%%%%%%%%%%%%%%%%%%%%%%%%%%%%%%%%%%%%%%%%%
%%%%%%%%%%%%%%%%%%%%%%%%%%%%%%%%%%%%%%%%%%%%%%%%%%%%%%%%%%%%%%%%%%%%%%%
\section{Preliminaries}
We first fix notation: in what follows $n\geq 2$ is always a positive integer and $m:=\lceil \tfrac{n}{2}\rceil$; that is, $n=2m$, if $n$ is even, and $n=2m-1$, if $n$ is odd. The symbols $a,b,c,x,y,\lambda$ will denote variables or real numbers. By capital calligraphic letters $\cA,\cM,\cK,\cL$ we denote matrices.  The symbols $G$ and $P$ are reserved for the functions defined in \cref{my_polynomials} and $M$ and $F$ denote the two hypergeometric functions defined in \cref{def_M} and \cref{def_hypergeom} below. The symbol $\langle\,,\,\rangle$ always denotes the inner product defined in \cref{inner_product}.
%%%%%%%%%%%%%%%%%%%%%%%%%%%%%%%%%%%%%%%%%%%%%%%%%%%%%%%%%%%%%%%%%%%%%%%
%%%%%%%%%%%%%%%%%%%%%%%%%%%%%%%%%%%%%%%%%%%%%%%%%%%%%%%%%%%%%%%%%%%%%%%
%%%%%%%%%%%%%%%%%%%%%%%%%%%%%%%%%%%%%%%%%%%%%%%%%%%%%%%%%%%%%%%%%%%%%%%
\subsection{Special functions}
Throughout the article we use a collection of special function. We present them in this subsection. The \emph{Pochhammer polynomials} \cite[18:3:1]{atlas} are defined by $(x)_n:=x(x+1)\cdots (x+n-1),$ where $n$ is a positive integer. If $n=0$, the definition is $(x)_0:=1$. \emph{Kummer's confluent hypergeometric function} \cite[Sec. 47]{atlas} is defined as
\begin{equation}\label{def_M}
M(a,c, x) := \sum_{k=0}^\infty \frac{(a)_k}{(c)_k}\, \frac{x^ k}{k!},
\end{equation}
and \emph{Gauss' hypergeometric function} \cite[Sec. 60]{atlas} is defined as
\begin{equation}\label{def_hypergeom}
F(a,b,c,x) := \sum_{k=0}^\infty \frac{(a)_k\,(b)_k }{(c)_k}\, \frac{x^ k}{k!},
\end{equation}
where $a,b,c\in\HR$, $c\neq 0,-1,-2,\ldots$. Generally,  neither $M(a,c,x)$ nor $F(a,b,c, x)$ converges for all $x$. But if either of the numeratorial parameters $a,b$ is a non-positive integer, both $M(a,c,x)$ and $F(a,b,c, x)$ reduce to polynomials and hence are defined for all $x\in\HR$ (and this is the only case we will meet throughout the paper).
\begin{rem}
Other common notations are $M(a,c,x)=~_1F_1(a;c;x)$ and $F(a,b,c,x)=~_1F_2(a,b;c;x)$. This is due to the fact that both $M(a,c,x)$ and $F(a,b,c,x)$ are special cases of the \emph{general hypergeometric function}  $~_qF_p(a_1,\ldots,a_q;c_1,\ldots,c_p;x):=\sum\limits_{k=0}^\infty \frac{\prod_{i=1}^p(a_i)_k}{\prod_{i=1}^q(c_i)_k}\, \frac{x^ k}{k!}$.
\end{rem}
The following will be useful later.
\begin{lemma}\label{lemma_hypergeom}
Let $a,b$ be non-positive integers and $c\neq 0,-1,-2,\ldots$. Then
	\[F(a,b+1,c,x)-F(a+1,b,c,x)= \frac{(a-b)x}{c} F(a+1,b+1,c+1,x)\]
\end{lemma}
\begin{proof}
Since $a$ and $b$ are non-negative integers, $F(a,b+1,c,x)$ and $F(a+1,b,c,x)$ are polynomials, whose constant term is equal to $1$. Therefore,
	\begin{equation}\label{401}
	F(a,b+1,c,x)-F(a+1,b,c,x)= \sum_{k=1}^\infty \frac{(a)_k(b+1)_k - (a+1)_k(b)_k}{(c)_k}\,\frac{x^k}{k!}.
	\end{equation}
 We have $
	(a)_k(b+1)_k - (b)_k(a+1)_k \stackrel{\text{\cite[18:5:6]{atlas}}}{=} (a)_k(b)_k(1+\tfrac{k}{b}) - (a)_k(b)_k(1+\tfrac{k}{a})= (a)_k(b)_k k \;\frac{a-b}{ab}.$
According to \cite[18:5:7]{atlas} the latter is equal to $(a+1)_{k-1}(b+1)_{k-1} k (a-b)$ and, moreover, $(c)_k=c(c+1)_{k-1}$. The claim follows when plugging this into \cref{401}.
\end{proof}
 For $x\geq 0$ the \emph{Gamma function} \cite[Sec. 43]{atlas} is defined as
 \begin{equation}\label{def_gamma}
 \Gamma(x):=\int_0^\infty t^{x-1} e^{-t} \d t.
 \end{equation}
The \emph{cumulative distribution function of the normal distribution} \cite[40:14:2]{atlas} and the \emph{error function} \cite[40:3:2]{atlas} are respectively defined as
	\begin{equation}\label{Phi}
	\Phi(x):=\frac{1}{\sqrt{2\pi}} \, \int_{-\infty}^\infty e^{-\frac{t^2}{2}}\, \d t,\quad \text{and}\quad \mathrm{erf}(x):= \frac{2}{\sqrt{\pi}} \; \int_0^x \exp(-t^2) \, \d t.
	\end{equation}
The error function and $\Phi(x)$ are related by the following equation	\cite[40:14:2]{atlas}
 	\begin{equation}
	2\Phi(x)=1+\mathrm{erf}\left(\frac{x}{\sqrt{2}}\right).\label{error_function_rel}
	\end{equation}
The error function and the Kummer's hypergeometric function are related by
		\begin{equation}
\mathrm{erf}(x)=\frac{2x}{\sqrt{\pi}}\,M\left(\tfrac{1}{2},\tfrac{3}{2},-x^2\right);  \label{error_function_rel2}
	\end{equation}
see \cite[13.6.19]{abramowitz}.
%%%%%%%%%%%%%%%%%%%%%%%%%%%%%%%%%%%%%%%%%%%%%%%%%%%%%%%%%%%%%%%%%%%%%%%
%%%%%%%%%%%%%%%%%%%%%%%%%%%%%%%%%%%%%%%%%%%%%%%%%%%%%%%%%%%%%%%%%%%%%%%
%%%%%%%%%%%%%%%%%%%%%%%%%%%%%%%%%%%%%%%%%%%%%%%%%%%%%%%%%%%%%%%%%%%%%%%
\subsection{Hermite polynomials}\label{sec:hermite}
Hermite polynomials are a family of polynomials $H_0(x),H_1(x),\ldots$ that are defined as
\begin{equation}
\label{hermite0}
H_n(x):=(-1)^n e^{x^2} \frac{\d^n}{\d x^n} e^{-x^2},
\end{equation}
see \cite[24:3:2]{atlas}.
An alternative Hermite function is defined by
	\begin{equation}\label{hermite}
	H_{e_n}(x):= (-1)^n e^{\tfrac{x^2}{2}} \frac{\d^n}{\d x^n} e^{-\tfrac{x^2}{2}}.
	\end{equation}
The two definitions are related by the following equality \cite[24:1:1]{atlas}
	\begin{equation}\label{hermite_relation}
	H_{e_n}(x)=\frac{1}{\sqrt{2}^{\, n}} H_n\left(\frac{x}{\sqrt{2}}\right).
	\end{equation}
By \cite[24:5:1]{atlas} we have that
	\begin{equation}\label{negative_hermite}
	H_k(-z)=(-1)^k H_k(z)\quad \text{and}\quad H_{e_k}(-z)=(-1)^k H_{e_k}(z)
	\end{equation}
	\begin{rem}
In the literature, the polynomials $H_n(x)$ are sometimes called the \emph{physicists' Hermite polynomials} and the $H_{e_n}(x)$ are sometimes called the \emph{probabilists' Hermite polynomials}. We will refer to both simply as \emph{Hermite polynomials} and distinguish them by using the respective symbols.
	\end{rem}
Hermite polynomials can be expressed in terms of Kummer's confluent hypergeometric function from \cref{def_M}:
\begin{align}\label{hermite_and_kummer}
H_{2k+1}(x)&=(-1)^k \frac{(2k+1)!\,2x}{k!}\,  M(-k,\tfrac{3}{2},x^2), \text{ and }\\
H_{2k}(x)&=(-1)^k \frac{(2k)!}{k!}\,  M(-k,\tfrac{1}{2},x^2);
\end{align}
see \cite[13.6.17 and 13.6.18]{abramowitz}.
%%%%%%%%%%%%%%%%%%%%%%%%%%%%%%%%%%%%%%%%%%%%%%%%%%%%%%%%%%%%%%%%%%%%%%%
%%%%%%%%%%%%%%%%%%%%%%%%%%%%%%%%%%%%%%%%%%%%%%%%%%%%%%%%%%%%%%%%%%%%%%%
%%%%%%%%%%%%%%%%%%%%%%%%%%%%%%%%%%%%%%%%%%%%%%%%%%%%%%%%%%%%%%%%%%%%%%%
\subsection{Orthogonality relations of Hermite polynomials}\label{sec:orth_rel}
The Hermite polynomials satisfy the following orthogonality relations. By \cite[7.374.2]{gradshteyn} we have
	\begin{equation}\label{important1}
	\int_\HR H_{e_m}(x)H_{e_n}(x) e^{-x^2} \d x =
	\begin{cases} (-1)^{ \lfloor\tfrac{m}{2}\rfloor+ \lfloor\tfrac{n}{2}\rfloor}  \;\Gamma\left(\frac{m+n+1}{2}\right), &\text{if $m+n$ is even}\\
	0,&\text{if $m+n$ is odd.} \end{cases},
	\end{equation}
where $\Gamma(x)$ is the Gamma function from \cref{def_gamma}. More generally, by \cite[p. 289, eq. (12)]{vol2}, if $m+n$ is even, we have  for $\alpha>0$, $\alpha^2\neq \tfrac{1}{2}$, that
\begin{equation}\label{important2}
\int_{-\infty}^\infty H_{e_m}(x)H_{e_n}(x) e^{-\alpha^2 x^2} \d x
 = \frac{(1-2\alpha^2)^{\frac{m+n}{2}}\,\Gamma\left(\frac{m+n+1}{2}\right)}{\alpha^{m+n+1}}  \, F\left(-m-n;\tfrac{1-m-n}{2};\tfrac{\alpha^2}{2\alpha^2-1}\right).
\end{equation}
Here $F(a,b,c,x)$ is Gauss' hypergeometric function as defined in \cref{def_hypergeom}. Recall from \cref{Phi} the definition of $\Phi(x)$. In the following we abbreviate
	\begin{equation}\label{my_polynomials}
	P_k(x):=\begin{cases} H_{e_k}(x),& \text{ if } k=0,1,2,\ldots\\
	-\sqrt{2\pi}\,e^{\frac{x^2}{2}}\Phi(x), &\text{ if } k=-1.
	\end{cases}
	\end{equation}
and put
\begin{equation}\label{G}
G_{k}(x):=\int_{-\infty}^x P_k(y)\, e^{-\tfrac{y^2}{2}} \;\d y, \;k=0,1,2,\ldots
\end{equation}
We can express the functions $G_k(x)$ in terms of the $P_k(x)$.
\begin{lemma}\label{lemma2}
 We have
\begin{enumerate}
\item For all $k$: $G_k(x) = - e^{-\tfrac{x^2}{2}}  P_{k-1}(x)$.
\item $G_k(\infty)=\begin{cases} \sqrt{2\pi}, &\text{ if } k=0\\ 0,&\text{ if } k\geq 1\end{cases}$
\end{enumerate}
\end{lemma}
\begin{proof}
Note that (2) is a direct consequence of (1). For (1) let $k\geq 0$ and write
	\[G_k(x)=\int_{y=-\infty}^x P_k(y) e^{-\tfrac{y^2}{2}}\,\d y \stackrel{\text{by \cref{my_polynomials}}}{=} \int_{y=-\infty}^x H_{e_k}(y) e^{-\tfrac{y^2}{2}}\,\d y \stackrel{\text{by \cref{hermite}}}{=} \int_{y=-\infty}^x (-1)^k \frac{\d^k}{\d y^k} e^{-\tfrac{y^2}{2}}\d y.\]
Thus $G_k(x)=(-1)^k \frac{\d^{k-1}}{\d x^{k-1}} e^{-\tfrac{x^2}{2}} = -e^{-\tfrac{x^2}{2}}P_{k-1}(x)$ as desired.
\end{proof}
We now fix the following notation: if two functions $f:\HR\to \HR$ and $g:\HR \to \HR$ satisfy $\int_\HR f(x)^2 e^{-\frac{x^2}{2}} \d x <\infty$ and $\int_\HR g(x)^2 e^{- \frac{x^2}{2}} \d x <\infty$, we define
	\begin{equation}\label{inner_product}
	\langle f(x),g(x) \rangle := \int_\HR f(x) g(x) e^{-\tfrac{x^2}{2}} \d x.
	\end{equation}
The Cauchy-Schwartz inequality implies $\langle f(x),g(x) \rangle<\infty$. The functions $P_k(x)$ and $G_k(x)$ satisfy the following orthogonality relations
\begin{lemma}\label{orthogonal_rel}
For all $k,\ell\geq 0$ with $k\neq \ell$ we have
\begin{enumerate}
\item $\langle G_k(x), P_{\ell}(x)\rangle = - \langle G_\ell(x), P_k(x)\rangle$.\\[0.01cm]
\item
$\langle G_k(x), P_\ell(x)\rangle =
\begin{cases}(-1)^{i+j} \Gamma\left(i+j-\tfrac{1}{2}\right), & \text{if } k=2i-1 \text{ and } \ell=2j\\
	(-1)^{i+j+1} \Gamma\left(i+j-\tfrac{1}{2}\right), & \text{if } k=2i \text{ and } \ell=2j-1\\
0, & \text{if } k+\ell \text{ is even}
\end{cases}
$
 \end{enumerate}
\end{lemma}
\begin{proof}
For (1) we have
\begin{align}\label{eqqqs}
\langle G_k(x), P_\ell(x)\rangle &=  \int_\HR G_k(x)P_\ell(x) e^{-\tfrac{x^2}{2}} \d x\\
	\nonumber &= \int_\HR \left(\int_{-\infty}^xP_k(y) e^{-\tfrac{y^2}{2}} \d y \right ) P_\ell(x) e^{-\tfrac{x^2}{2}} \d x\\
	\nonumber &=  \int_\HR \left(\int_{y}^\infty P_\ell(x)e^{-\tfrac{x^2}{2}} \d x \right )P_k(y) e^{-\tfrac{y^2}{2}} \d y\\
	\nonumber &=  (-1)^\ell \int_\HR \left(\int_{-\infty}^{-y} P_\ell(x)e^{-\tfrac{x^2}{2}} \d x \right )P_k(y) e^{-\tfrac{y^2}{2}} \d y\\
	\nonumber  &=  (-1)^{k+\ell} \int_\HR \left(\int_{-\infty}^y P_\ell(x)e^{-\tfrac{x^2}{2}} \d x \right )P_k(y)  e^{-\tfrac{y^2}{2}} \d y\\
	  	\nonumber &=  (-1)^{k+\ell}\langle G_\ell(x), P_k(x)\rangle,
	\end{align}
where the fourth equality is due to the transformation $x\mapsto -x$ and equation \cref{negative_hermite} and the fifth equality is obtained using the transformation $y\mapsto -y$.  This shows (1) for the case $k+\ell$ odd. The case $k+\ell$ even is implied by (2), which we prove next.

Since $k$ and $\ell$ are not both zero, by \cref{eqqqs}, we may assume that $k>0$. In this case,  by \cref{lemma2}, we have $G_k(x) = - P_{k-1}(x) e^{-\frac{x^2}{2}}$, so that
	\begin{align}
	\langle G_k(x), P_\ell(x)\rangle  &= -  \int_\HR P_{k-1}(x)P_\ell(x) e^{-x^2} \d x \nonumber
	 =  - \int_\HR H_{e_{k-1}}(x)H_{e_{\ell}}(x) e^{-x^2} \d x.\label{t1}
	\end{align}
Combining this equation with \cref{important1}, we have
	$$\langle G_k(x), P_\ell(x)\rangle =
	\begin{cases} (-1)^{ \lfloor\tfrac{k-1}{2}\rfloor+ \lfloor\tfrac{\ell}{2}\rfloor + 1} \, \Gamma\left(\frac{k+\ell}{2}\right), &\text{if $k+\ell-1$ is even}\\
	0,&\text{if $k+\ell-1$ is odd} \end{cases}
	$$
In particular, $\langle G_k(x), P_\ell(x)\rangle = 0$ for $k+\ell$ even, which finishes the proof of the first part of this lemma. The second part is proved by replacing $k=2i-1$ and $\ell = 2j$. The case $k=2i$ and $\ell = 2j-1$ is a consequence of the case $k=2i-1$ and $\ell = 2j$ and the first part of the theorem (we can't prove this last case simply by plugging in, because $k=2i$ might violate the assumption~$k>0$).
This finishes the proof.
\end{proof}
%%%%%%%%%%%%%%%%%%%%%%%%%%%%%%%%%%%%%%%%%%%%%%%%%%%%%%%%%%%%%%%%%%%%%%%
%%%%%%%%%%%%%%%%%%%%%%%%%%%%%%%%%%%%%%%%%%%%%%%%%%%%%%%%%%%%%%%%%%%%%%%
%%%%%%%%%%%%%%%%%%%%%%%%%%%%%%%%%%%%%%%%%%%%%%%%%%%%%%%%%%%%%%%%%%%%%%%
\subsection{The expected value of Hermite polynomials}
In this section we will compute the expected value of the Hermite polynomials when the argument follows a normal distribution.
\begin{lemma}\label{expectation_hermite1}
For $\sigma^2>0$ we have  $\mean\limits_{u\sim N(0,\sigma^2)} H_{{2k}}(u)=\frac{(2k)!}{k!}\, (2\sigma^2 -1 )^k$.
\end{lemma}
	\begin{proof}
Write
	\begin{equation*}
	\mean\limits_{u\sim N(0,\sigma^2)} H_{{2k}}(u) \stackrel{\text{by definition}}{=}
	\frac{1}{\sqrt{2\pi\sigma^2}} \int_{u=-\infty}^\infty H_{2k}(u)\, e^{-\tfrac{u^2}{2\sigma^2}} \d u = \frac{1}{\sqrt{\pi}} \int_{w=-\infty}^\infty H_{2k}(\sqrt{2\sigma^2}\, w) \, e^{-w^2} \d w,
	\end{equation*}
	where the second equality is due to the change of variables $w := \tfrac{u}{\sqrt{2\sigma^2}}$. Applying \cite[7.373.2]{gradshteyn} we get
	\[\frac{1}{\sqrt{\pi}} \int_{w=-\infty}^\infty H_{2k}(\sqrt{2\sigma^2}\, w) e^{-w^2} \d w= \frac{(2k)!\, (2\sigma^2 -1 )^k}{k!} .\]
This finishes the proof.
\end{proof}
\begin{lemma}\label{expectation_hermite2}
Let $\sigma^2>0$ and recall from \cref{my_polynomials} the definition of $P_k(x)$, $k=-1,0,1,2,\ldots$.
\begin{enumerate}
\item If $k,\ell>0$ and $k+\ell$ is even, we have
\[\mean\limits_{u\sim N(0,\sigma^2)} P_k(u)P_\ell(u) e^{-\tfrac{u^2}{2}} = \frac{(-1)^{\frac{k+\ell}{2}}\; \sqrt{2}^{k+\ell}\, \Gamma\left(\frac{k+\ell+1}{2}\right)}{\, \sqrt{\pi} \;\sqrt{\sigma^2+1}^{\,k+\ell+1}}  \, F\left(-k,-\ell;\tfrac{1-k-\ell}{2};\tfrac{\sigma^2+1}{2}\right).\]
\item For all $k$ we have
\[\mean\limits_{u\sim N(0,\sigma^2)} P_{-1}(u)P_{2k+1}(u) e^{-\tfrac{u^2}{2}} = \frac{(-1)^{k+1} (2k+1)!}{2^{k}\,k!}\,\frac{(1-\sigma^2)^k\sigma^{2}}{\sqrt{1+\sigma^{2}} }\;F\left(-k,\tfrac{1}{2},\tfrac{3}{2},\tfrac{\sigma^4}{\sigma^4-1}\right).\]
\end{enumerate}
\end{lemma}
\begin{proof}
To prove (1) we write
\begin{align*}
	\mean\limits_{u\sim N(0,\sigma^2)} P_k(u)P_\ell(u) e^{-\tfrac{u^2}{2}} &=\mean\limits_{u\sim N(0,\sigma^2)} H_{e_{k}}(u) H_{e_\ell}(u) e^{-\tfrac{u^2}{2}}\\
	&= \frac{1}{\sqrt{2\pi\sigma^2}}  \int_{u=-\infty}^\infty H_{e_{k}}(u) H_{e_\ell}(u) e^{-\tfrac{u^2}{2}\left(1+\tfrac{1}{\sigma^2}\right)}  \d u.
\end{align*}
Put $\alpha^2:=\tfrac{1}{2}\left(1+\tfrac{1}{\sigma^2}\right)$ and observe that $\alpha^2\neq \tfrac{1}{2}$. By \cref{important2} we have
\begin{align*}
&\frac{1}{\sqrt{2\pi\sigma^2}} \, \int_{u=-\infty}^\infty \,H_{e_{k}}(u) H_{e_\ell}(u) e^{-\tfrac{u^2}{2}\left(1+\tfrac{1}{\sigma^2}\right)} \, \d u\\
 =  \quad &\frac{(1-2\alpha^2)^{\frac{k+\ell}{2}}\,\Gamma\left(\frac{k+\ell+1}{2}\right)}{\sqrt{2\pi\sigma^2} \;\alpha^{k+\ell+1}}  \, F\left(-k,-\ell;\tfrac{1-k-\ell}{2};\tfrac{\alpha^2}{2\alpha^2-1}\right)\\
 =  \quad &\frac{(-1)^{\frac{k+\ell}{2}}\; \sqrt{2}^{k+\ell}\, \Gamma\left(\frac{k+\ell+1}{2}\right)}{\, \sqrt{\pi} \;\sqrt{\sigma^2+1}^{\,k+\ell+1}}  \, F\left(-k,-\ell;\tfrac{1-k-\ell}{2};\tfrac{\sigma^2+1}{2}\right)
\end{align*}
This proves (1). For (2) we have
\begin{align*}
\mean\limits_{u\sim N(0,\sigma^2)} P_{-1}(u)P_{{2k+1}}(u) e^{-\tfrac{u^2}{2}} &=-\sqrt{2\pi}\,\mean\limits_{u\sim N(0,\sigma^2)} \Phi(u)H_{e_{2k+1}}(u) \\
&= \frac{-1}{\sigma} \, \int_{u=-\infty}^\infty \,\Phi(u) H_{e_{2k+1}}(u) \,e^{-\tfrac{u^2}{2\sigma^2}} \, \d u.
 \end{align*}
Making a change of variables $x:=\tfrac{u}{\sqrt{2}}$ the last integral becomes
	\begin{equation*}
	\frac{-1}{\sqrt{2}\, \sigma} \, \int_{u=-\infty}^\infty \,\Phi(\sqrt{2}\,x) H_{e_{2k+1}}(\sqrt{2}\,x) \,e^{-\tfrac{x^2}{\sigma^2}} \, \d x.
	\end{equation*}
Using \cref{error_function_rel} and \cref{hermite_relation} we write this integral as
$$\frac{-1}{2^{k+1}\,\sigma} \, \int_{x=-\infty}^\infty \,(1+\mathrm{erf}(x)) H_{{2k+1}}(x) \,e^{-\tfrac{x^2}{\sigma^2}} \, \d x.$$
We know from \cref{negative_hermite} that $H_{2k+1}(x)$ is an odd function, which implies that $$\int_{x=-\infty}^\infty H_{2k+1}(x) e^{-\tfrac{x^2}{\sigma^2}} \, \d x=0.$$ Moreover, by \cref{hermite_and_kummer} we have $H_{2k+1}(x)=(-1)^k \frac{(2k+1)!\,2x}{k!}\,  M(-k,\tfrac{3}{2},x^2)$ and by \cref{error_function_rel2} we have $\mathrm{erf}(x)=\frac{2x}{\sqrt{\pi}}\,M\left(\tfrac{1}{2},\tfrac{3}{2},-x^2\right)$. All this shows that
	\begin{align}
	\mean\limits_{u\sim N(0,\sigma^2)} P_{-1}(u)P_{{2k+1}}(u) e^{-\tfrac{u^2}{2}}&=\frac{(-1)^{k+1} (2k+1)!}{2^{k-1}\,\sqrt{\pi}\,\sigma\, k!} \, \int_{x=-\infty}^\infty \,x^2\,M\left(\tfrac{1}{2},\tfrac{3}{2},-x^2\right)\,M(-k,\tfrac{3}{2},x^2) \,e^{-\tfrac{x^2}{\sigma^2}} \, \d x \nonumber\\
	&=\frac{(-1)^{k+1} (2k+1)!}{2^{k-2}\,\sqrt{\pi}\,\sigma\, k!} \, \int_{x=0}^\infty \,x^2\,M\left(\tfrac{1}{2},\tfrac{3}{2},-x^2\right)\,M(-k,\tfrac{3}{2},x^2) \,e^{-\tfrac{x^2}{\sigma^2}} \, \d x.\label{105}
	\end{align}
where for the second equality we used that the integrand is an even function. Making a change of variables $t:=x^2$ we see that
	\begin{equation*}
	\int_{x=0}^\infty \,x^2\,M\left(\tfrac{1}{2},\tfrac{3}{2},-x^2\right)\,M(-k,\tfrac{3}{2},x^2) \,e^{-\tfrac{x^2}{\sigma^2}} \, \d x=\frac{1}{2}\int_{t=0}^\infty \,\sqrt{t}\;M\left(\tfrac{1}{2},\tfrac{3}{2},-t\right)\,M(-k,\tfrac{3}{2},t) \,e^{-\tfrac{t}{\sigma^2}} \, \d t.
\end{equation*}
By \cite[7.622.1]{gradshteyn} we have
	\begin{equation*}
	\int_{t=0}^\infty \,\sqrt{t}\;M\left(\tfrac{1}{2},\tfrac{3}{2},-t\right)\,M(-k,\tfrac{3}{2},t) \,e^{-\tfrac{t}{\sigma^2}} \, \d t=\Gamma\left(\frac{3}{2}\right)\,\frac{(1-\sigma^2)^k\sigma^{3}}{\sqrt{1+\sigma^{2}} }\;F\left(-k,\tfrac{1}{2},\tfrac{3}{2},\tfrac{\sigma^4}{\sigma^4-1}\right)
\end{equation*}
Plugging the last two equations into \cref{105} we obtain
\begin{align*}
\mean\limits_{u\sim N(0,\sigma^2)} P_{-1}(u)P_{{2k+1}}(u) e^{-\tfrac{u^2}{2}}&=\frac{(-1)^{k+1} (2k+1)!}{2^{k-1}\,\sqrt{\pi}\,\sigma\, k!}\;\Gamma\left(\tfrac{3}{2}\right)\,\frac{(1-\sigma^2)^k\sigma^{3}}{\sqrt{1+\sigma^{2}} }\;F\left(-k,\tfrac{1}{2},\tfrac{3}{2},\tfrac{\sigma^4}{\sigma^4-1}\right)\\
&=\frac{(-1)^{k+1} (2k+1)!}{2^{k}\,k!}\,\frac{(1-\sigma^2)^k\sigma^{2}}{\sqrt{1+\sigma^{2}} }\;F\left(-k,\tfrac{1}{2},\tfrac{3}{2},\tfrac{\sigma^4}{\sigma^4-1}\right).
\end{align*}
For the second equality we have used $\Gamma\left(\tfrac{3}{2}\right)=\tfrac{\sqrt{\pi}}{2}$, see \cite[43:4:3]{atlas}. This finishes the proof.
	\end{proof}
%%%%%%%%%%%%%%%%%%%%%%%%%%%%%%%%%%%%%%%%%%%%%%%%%%%%%%%%%%%%%%%%%%%%%%%
%%%%%%%%%%%%%%%%%%%%%%%%%%%%%%%%%%%%%%%%%%%%%%%%%%%%%%%%%%%%%%%%%%%%%%%
%%%%%%%%%%%%%%%%%%%%%%%%%%%%%%%%%%%%%%%%%%%%%%%%%%%%%%%%%%%%%%%%%%%%%%%
%%%%%%%%%%%%%%%%%%%%%%%%%%%%%%%%%%%%%%%%%%%%%%%%%%%%%%%%%%%%%%%%%%%%%%%
\section{Proof of Theorem \ref{cor}}\label{sec:proof2}
In this section we prove \cref{cor}. In \cite[Theorem 4.3]{draisma-horobet} Draisma and Horobet present a formula for $E(n,p)$ in terms of the expected modulus of the characteristic polynomial of a GOE-matrix:
	\begin{equation*}
	E(n,p)=\frac{\sqrt{\pi}}{\sqrt{2}^{\,n-1}\,\Gamma(\tfrac{n}{2})}\;\mean\limits_{\substack{w\sim N(0,1)\\ A\sim \mathrm{GOE}(n-1)}} \left\lvert \det\left(\sqrt{2(p-1)}\,A-\sqrt{p}\,w I_{n-1} \right)\right\rvert,
\end{equation*}
where $I_{n-1}$ is the $(n-1)\times(n-1)$ identity matrix.
Observe that
	\begin{equation*}\label{E-eq2}
	\mean\limits_{A\sim \mathrm{GOE}(n-1)} \left\lvert \det\big(\sqrt{2(p-1)}\,A-\sqrt{p}\,w I_{n-1} \big)\right\rvert = \sqrt{2(p-1)}^{n-1}\mean\limits_{A\sim \mathrm{GOE}(n-1)} \left\lvert \det(A-uI_{n-1})\right\rvert,
\end{equation*}
where $ u= \sqrt{\tfrac{p}{2(p-1)}}\, w$. Using the notation from \cref{def_I_J} we therefore have
\begin{equation}\label{100}
E(n,p)=\frac{\sqrt{\pi}\sqrt{p-1}^{\,n-1}}{\Gamma(\tfrac{n}{2})}\mean\limits_{u\sim N(0,\sigma^2)} \cI_{n-1}(u),   \text{ with } \sigma^2=\frac{p}{2(p-1)}.
\end{equation}
We now have to distinguish between the cases $n$ even and $n$ odd. The distinction between those cases is due to the nature of \cref{thm}: The formula for $\cI_{n-1}(u)$ depends on the parity of $n$.
%%%%%%%%%%%%%%%%%%%%%%%%%%%%%%%%%%%%%%%%%%%%%%%%%%%%%%%%%%%%%%%%%%%%%%%
%%%%%%%%%%%%%%%%%%%%%%%%%%%%%%%%%%%%%%%%%%%%%%%%%%%%%%%%%%%%%%%%%%%%%%%
%%%%%%%%%%%%%%%%%%%%%%%%%%%%%%%%%%%%%%%%%%%%%%%%%%%%%%%%%%%%%%%%%%%%%%%
\subsection{Proof of Theorem \ref{cor} (1)} In this case we have $n=2m+1$ and hence $n-1=2m$, so that \cref{100} becomes
$$E(2m+1,p) = \frac{\sqrt{\pi}\,(p-1)^{m}}{\Gamma(\tfrac{2m+1}{2})}\mean\limits_{u\sim N(0,\sigma^2)} \cI_{2m}(u),  \text{ where } \sigma^2=\frac{p}{2(p-1)}.$$
We know from \cref{thm} (1) that
\begin{align}\label{302}
\cI_{2m}(u)=&\cJ_{2m}(u)+\frac{\sqrt{2\pi} \,e^{-\tfrac{u^2}{2}}}{\prod_{i=1}^{2m}\Gamma\left(\tfrac{i}{2}\right)}\,\sum_{1\leq i,j \leq m} \det(\Gamma_1^{i,j}) \; \det\begin{bmatrix} P_{2i-1}(u) & P_{2j}(u) \\ P_{2i-2}(u) & P_{2j-1}(u) \end{bmatrix},
\end{align}
For taking the expectation of $\cI_{2m}(u)$ over $u$ we may take the expectation over the two summands separately. The first expectation is given by the following lemma. We will prove it in \cref{proof_lem_301} below.
\begin{lemma}\label{301} We have
$\frac{\sqrt{\pi}\,(p-1)^m}{\Gamma(\tfrac{2m+1}{2})}\;\mean\limits_{u\sim N(0,\sigma^2)} \cJ_{2m}(u) = 1$.
\end{lemma}
The second expectation is given by the following lemma. We will prove it in \cref{proof_lem_10}.
\begin{lemma}\label{lemma10}
Recall from \cref{cor} the definition for $j>0$:
	\begin{equation*}
	g_{i,j}(p) = \frac{\Gamma\left(i+j-\frac{1}{2}\right)}{\tfrac{1-2i-2j}{1-2i+2j}\left(-\tfrac{3p-2}{4(p-1)}\right)^{i+j-1}} \,  F\left(-2i,1-2j,\tfrac{3}{2}-i-j,\tfrac{3p-2}{4(p-1)}\right)
	\end{equation*}
For any $1\leq i,j\leq m$ we have
	\begin{equation*}
	 \mean\limits_{u\sim N(0,\sigma^2)}e^{-\tfrac{u^2}{2}}\det\begin{bmatrix} P_{2i-1}(u) & P_{2j}(u) \\ P_{2i-2}(u) & P_{2j-1}(u) \end{bmatrix}\\
	= \frac{1}{\sqrt{2\pi}}\, \sqrt{\frac{3p-2}{p-1}}\,g_{i-1,j}(p).
\end{equation*}
\end{lemma}
\cref{301} and \cref{lemma10} in combination with \cref{100} and \cref{302} show that $E(2m+1,p)$ equals
	\begin{align*}
	E(2m+1,p) &=1+\frac{\sqrt{\pi}\,(p-1)^{m-1}\sqrt{(p-1)(3p-2)}}{\prod_{i=1}^{2m+1}\Gamma\left(\tfrac{i}{2}\right)}\sum_{1\leq i,j\leq m} \det(\Gamma_1^{i,j}) \, g_{i-1,j}(p),
	\end{align*}
which shows \cref{cor} (1).\qed

%%%%%%%%%%%%%%%%%%%%%%%%%%%%%%%%%%%%%%%%%%%%%%%%%%%%%%%%%%%%%%%%%%%%%%%
%%%%%%%%%%%%%%%%%%%%%%%%%%%%%%%%%%%%%%%%%%%%%%%%%%%%%%%%%%%%%%%%%%%%%%%
%%%%%%%%%%%%%%%%%%%%%%%%%%%%%%%%%%%%%%%%%%%%%%%%%%%%%%%%%%%%%%%%%%%%%%%
\subsection{The case $n=2m$ is even} In this case we have $n-1=2m-1$, so that \cref{100} becomes
\begin{equation*}E(2m,p)=\frac{\sqrt{\pi}\sqrt{p-1}^{\,2m-1}}{\Gamma(2m)}\mean\limits_{u\sim N(0,\sigma^2)}\cI_{2m-1}(u),\quad \text{where } \sigma^2=\frac{p}{2(p-1)}.
\end{equation*}
We apply \cref{thm} (2) to obtain
\[\cI_{2m-1}(u)=\cJ_{2m-1}(u)  +  \frac{\sqrt{2}\,e^{-\tfrac{u^2}{2}}}{\prod_{i=1}^{2m-1}\Gamma\left(\tfrac{i}{2}\right)}\,\sum_{0\leq i,j\leq m-1}  \det(\Gamma_2^{i,j})\,  \det\begin{bmatrix} P_{2j}(u) &   P_{2i+1}(u)\\   P_{2j-1}(u) & P_{2i}(u)\end{bmatrix}.\]
Since the normal distribution is symmetric around the origin we have
	\begin{align}
	\mean\limits_{u\sim N(0,\sigma^2)} \cJ_{2m-1}(u)&=\mean\limits_{u\sim N(0,\sigma^2)}\mean\limits_{A\sim \mathrm{GOE}(2m-1)}  \det(A-u I_{2m-1})\label{exp_J0}\\
	&=(-1)^{2m-1}\mean\limits_{u\sim N(0,\sigma^2)}\mean\limits_{A\sim \mathrm{GOE}(2m-1)}  \det((-A)-(-u) I_{2m-1})\nonumber\\
	&= -\mean\limits_{u\sim N(0,\sigma^2)}\mean\limits_{A\sim \mathrm{GOE}(2m-1)}  \det(A-u I_{2m-1})\nonumber\\
	 &= -\mean\limits_{u\sim N(0,\sigma^2)} \cJ_{2m-1}(u),\nonumber
	\end{align}
and hence $\mean\limits_{u\sim N(0,\sigma^2)} \cJ_{2m-1}(u)=0$. This shows that
\begin{equation*}
E(2m,p)=\frac{\sqrt{2\pi}\sqrt{p-1}^{\,2m-1}}{\prod_{i=1}^{2m}\Gamma\left(\tfrac{i}{2}\right)}\sum_{0\leq i,j\leq m-1}  \det(\Gamma_2^{i,j}) \mean\limits_{u\sim N(0,\sigma^2)}e^\frac{-u^2}{2}  \det\begin{bmatrix} P_{2j}(u) &   P_{2i+1}(u)\\   P_{2j-1}(u) & P_{2i}(u)\end{bmatrix}.
\end{equation*}
The next lemma gives a formula for the expected values under the summation sign in this equation. We prove it in \cref{proof_lem_3} below.
\begin{lemma}\label{lemma3}
Recall from \cref{cor} the definition
	\begin{align*}
	g_{i,j}(p) = \begin{cases}
	\displaystyle\frac{\sqrt{\pi}(2i+1)!}{(-1)^{i} 2^{2i}\,i!}\,\frac{(p-2)^ip}{(p-1)^{i}(3p-2)} \;F\left(-i,\tfrac{1}{2},\tfrac{3}{2},\tfrac{-p^2}{(3p-2)(p-2)}\right) - \frac{   \Gamma\left(j+\frac{1}{2}\right)}{2\,\left(-\frac{3p-2}{4(p-1)}\right)^{j+1}}, &\text{ if } j =0 \\[0.2cm]
	\displaystyle\frac{ \,\Gamma\left(i+j+\frac{1}{2}\right)}{\tfrac{(1-2i-2j)}{(1-2i+2j)}\,\left(-\tfrac{3p-2}{4(p-1)}\right)^{i+j} } F\left(-2i,1-2j,\tfrac{3}{2}-i-j,\tfrac{3p-2}{4(p-1)}\right), &\text{ if } j >0. \end{cases}
	\end{align*}
For all $0\leq i\leq m-1$ and $0\leq j\leq m-1$ the following holds.
	\begin{align*}
\mean\limits_{u\sim N(0,\sigma^2)}e^\frac{-u^2}{2}\det\begin{bmatrix} P_{2j}(u) &   P_{2i+1}(u)\\   P_{2j-1}(u) & P_{2i}(u)\end{bmatrix} = \frac{1}{\sqrt{2\pi}} \,\sqrt{\frac{3p-2}{p-1}}\, g_{i,j}(p).
\end{align*}
\end{lemma}
The lemma implies that
\begin{equation*}
E(2m,p)=\frac{(p-1)^{m-1} \,\sqrt{(p-1)(3p-2)}}{\prod_{i=1}^{2m}\Gamma\left(\tfrac{i}{2}\right)}\sum_{0\leq i,j\leq m-1}  \det(\Gamma_2^{i,j}) g_{i,j}(p).
\end{equation*}
which proves \cref{cor} (2).\qed

\subsection{Proofs of the lemmata}
Before proving the lemmata from the previous subsections, we will have to prove the following technical lemma.
\begin{lemma}\label{301.1} For all $m\geq 1$ we have
\begin{enumerate}
\item $\sqrt{\pi}\, (2(m-1))! (2m-1)=2^{2m-1} \;\Gamma\left(\tfrac{2m+1}{2}\right)\;\Gamma(m)$.
\item $\sqrt{\pi}^{\, m+1}\,\left[\prod_{i=1}^{m-1} (2i)!\right](2m)!=m!2^{\,m(m+1)} \; \prod_{i=1}^{2m+1}\Gamma\left(\tfrac{i}{2}\right)$.
\end{enumerate}
\end{lemma}
\begin{proof} We will use the identities $\Gamma(\tfrac{1}{2})=\sqrt{\pi}$ and $\Gamma(\tfrac{3}{2})=\tfrac{\sqrt{\pi}}{2}$ \cite[43:4:3]{atlas} and $\Gamma(x+1)=x\Gamma(x)$ for $x>0$ \cite[43:4:3]{atlas}. We prove both claims using an induction argument. For (1) and $m=1$ we have
	\[\frac{\sqrt{\pi}\, (2(m-1))! (2m-1)}{2^{2m-1} \;\Gamma\left(\tfrac{2m+1}{2}\right)\;\Gamma(m)} = \frac{\sqrt{\pi}}{\sqrt{\pi}} = 1.\]
For $m>1$, using the induction hypothesis, we have
	\begin{align*}\frac{\sqrt{\pi}\, (2(m-1))! (2m-1)}{2^{2m-1}\;\Gamma\left(\tfrac{2m+1}{2}\right)\;\Gamma(m)}
	 = &\frac{(2m-2)(2m-3) }{4} \frac{2m-1}{2m-3} \frac{\Gamma\left(\tfrac{2m-1}{2}\right)}{\Gamma\left(\tfrac{2m+1}{2}\right)} \frac{\Gamma(m-1)}{\Gamma(m)}\\
	  = &\frac{(2m-2)(2m-1) }{4} \frac{2}{2m-1} \frac{1}{m-1} = 1
	 \end{align*}
For (2) and $m=1$ we have
	\[\sqrt{\pi}^{\, m+1}\,\left[\prod_{i=1}^{m-1} (2i)!\right](2m)!=2\pi  = m!2^{\,m(m+1)} \; \prod_{i=1}^{2m+1}\Gamma\left(\tfrac{i}{2}\right)\]
For $m>1$, using the induction hypothesis, we have
	\begin{align*}\frac{\sqrt{\pi}^{\, m+1}\,\left[\prod_{i=1}^{m-1} (2i)!\right](2m)!}{m!2^{\,m(m+1)} \; \prod_{i=1}^{2m+1}\Gamma\left(\tfrac{i}{2}\right)}
	 = &\frac{\sqrt{\pi}\,(2(m-1))! 2m (2m-1)}{m2^{2m} \; \Gamma\left(\tfrac{2m+1}{2}\right) \;\Gamma(m)} = \frac{\sqrt{\pi}\,(2(m-1))! (2m-1)}{2^{2m-1} \; \Gamma\left(\tfrac{2m+1}{2}\right) \;\Gamma(m)}=1,
	 \end{align*}
the last equality because of (1).  This finishes the proof.
\end{proof}
\subsubsection{Proof of Lemma \ref{301}}\label{proof_lem_301}
The formula from \cite[Eq. (22.2.38)]{mehta} implies that for $n=2m$:
	\[\cJ_{2m}(u) = \frac{ \sqrt{\pi}^{\, m}\, \left[\prod_{i=1}^{m-1} (2i)!\right]}{2^{\,m(m+1)} \; \prod_{i=1}^{2m}\Gamma\left(\tfrac{i}{2}\right)} \, H_{2m}(u).\]
For odd $n$ a similar but more involved formula can found in \cite[Eq. (22.2.39)]{mehta}. However, we decided not to put it here, because we use the symmetry argument in \cref{exp_J0}.

By \cref{expectation_hermite1} (1) we have $\mean_{u\sim N(0,\sigma^2)} H_{{2m}}( u)=\frac{(2m)!}{m!} (2\sigma^2-1)^m$. Thus,
\[\mean\limits_{u\sim N(0,\sigma^2)} H_{{2m}}( u)=\frac{(2m)!}{m!(p-1)^m} \;\text{ for }\; \sigma^2= \frac{p}{2(p-1)}.\]
This implies
	\begin{align*}
	\frac{\sqrt{\pi}\,(p-1)^{m}}{\Gamma(\tfrac{2m+1}{2})}\;\mean\limits_{u\sim N(0,\sigma^2)} \cJ_{2m}(u)
	&= \frac{\sqrt{\pi}\,(p-1)^{m}}{\Gamma(\tfrac{2m+1}{2})}\frac{ \sqrt{\pi}^{\, m}\, \left[\prod_{i=1}^{m-1} (2i)!\right]}{2^{\,m(m+1)} \; \prod_{i=1}^{2m}\Gamma\left(\tfrac{i}{2}\right)}\frac{(2m)!}{m!(p-1)^m}\\
	&= \frac{ \sqrt{\pi}^{\, m+1}\, \left[\prod_{i=1}^{m-1} (2i)!\right]}{2^{\,m(m+1)} \; \prod_{i=1}^n\Gamma\left(\tfrac{i}{2}\right)}\frac{(2m)!}{m!} =1
	\end{align*}
the last equality by \cref{301.1} (2). \qed

\subsubsection{Proof of Lemma \ref{lemma10}}\label{proof_lem_10}
From \cref{expectation_hermite1} (2) we get for all $1\leq i\leq m$ that
\begin{align}\label{505}
 &\mean\limits_{u\sim N(0,\sigma^2)} P_{{2i-1}}(u)P_{{2j-1}}(u)\,e^{-\tfrac{u^2}{2}}\\
 =&\frac{(-1)^{i+j-1}\; 2^{i+j-1}\, \Gamma\left(i+j-1+\frac{1}{2}\right)}{\, \sqrt{\pi} \;(\sigma^2+1)^{i+j-1+\tfrac{1}{2}}}  F\left(1-2i,1-2j;\tfrac{1}{2}-i-j+1;\tfrac{3p-2}{4(p-1)}\right)\nonumber\\
  = &\frac{(-1)^{i+j-1} 4^{i+j}\, \Gamma\left(i+j-\frac{1}{2}\right)\,}{2\sqrt{2\pi}} \; \left(\frac{p-1}{3p-2}\right)^{i+j-\tfrac{1}{2}}  F\left(1-2i,1-2j;\tfrac{3}{2}-i-j;\tfrac{3p-2}{4(p-1)}\right),\nonumber
 \end{align}
and
\begin{align}\label{502}
&\mean\limits_{u\sim N(0,\sigma^2)} P_{{2i}}(u)P_{{2j-2}}(u)\,e^{-\tfrac{u^2}{2}}\\
 =&\frac{(-1)^{i+j-1}\; 2^{i+j-1}\, \Gamma\left(i+j-1+\frac{1}{2}\right)}{\, \sqrt{\pi} \;(\sigma^2+1)^{i+j-1+\tfrac{1}{2}}}  F\left(-2i,2-2j;\tfrac{1}{2}-i-j+1;\tfrac{3p-2}{4(p-1)}\right)\nonumber\\
  = &\frac{(-1)^{i+j-1} 4^{i+j}\, \Gamma\left(i+j-\frac{1}{2}\right)\,}{2\sqrt{2\pi}} \; \left(\frac{p-1}{3p-2}\right)^{i+j-\tfrac{1}{2}}  F\left(-2i,2-2j;\tfrac{3}{2}-i-j;\tfrac{3p-2}{4(p-1)}\right).\nonumber
 \end{align}
By \cref{lemma_hypergeom} we have
	\begin{align*}
	&F\left(1-2i,1-2j,\tfrac{3}{2}-i-j,\tfrac{3p-2}{4(p-1)}\right)- F\left(2-2i,-2j,\tfrac{3}{2}-i-j,\tfrac{3p-2}{4(p-1)}\right)\\
	=\;&2\,\frac{3p-2}{4(p-1)}\,\frac{1-2i+2j}{3-2i-2j} F\left(2-2i,1-2j,\tfrac{5}{2}-i-j,\tfrac{3p-2}{4(p-1)}\right).
	\end{align*}
This shows that
	\begin{align*}
	 &\mean\limits_{u\sim N(0,\sigma^2)}e^{-\tfrac{u^2}{2}}\det\begin{bmatrix} P_{2i-1}(u) & P_{2j}(u) \\ P_{2i-2}(u) & P_{2j-1}(u) \end{bmatrix}\\
	 =&\mean\limits_{u\sim N(0,\sigma^2)}e^{-\tfrac{u^2}{2}}\,(P_{2i-1}(u)P_{2j-1}(u) - P_{2i-2}(u)P_{2j}(u))\\
	=&  \frac{\Gamma\left(i+j-\frac{1}{2}\right)}{\sqrt{2\pi}\,\tfrac{3-2i-2j}{1-2i+2j}\,\left(-\tfrac{3p-2}{4(p-1)}\right)^{i+j-1}} \, \sqrt{\frac{3p-2}{p-1}}\; F\left(2-2i,1-2j,\tfrac{5}{2}-i-j,\tfrac{3p-2}{4(p-1)}\right)\\
	=&\frac{1}{\sqrt{2\pi}}\,\sqrt{\frac{3p-2}{p-1}}\,g_1^{i-1,j}(p),
	 \end{align*}
	which finishes the proof. \qed

\subsubsection{Proof of Lemma \ref{lemma3}}\label{proof_lem_3}
First, we prove the case $j>0$: writing
\begin{align*}
\det\begin{bmatrix} P_{2j}(u) &   P_{2i+1}(u)\\   P_{2j-1}(u) & P_{2i}(u)\end{bmatrix} = \det\begin{bmatrix} P_{2j}(u) &   P_{2(i+1)-1}(u)\\   P_{2j-1}(u) & P_{2(i+1)-2}(u)\end{bmatrix}
\end{align*}
we see that \cref{lemma10} implies
\begin{align*}
\mean\limits_{u\sim N(0,\sigma^2)}e^\frac{-u^2}{2}\det\begin{bmatrix} P_{2j}(u) &   P_{2i+1}(u)\\   P_{2j-1}(u) & P_{2i}(u)\end{bmatrix} &= \frac{1}{\sqrt{2\pi}} \,\sqrt{\frac{3p-2}{p-1}}\, g_{(i+1)-1,j}(p)\\
&=\frac{1}{\sqrt{2\pi}} \,\sqrt{\frac{3p-2}{p-1}}\, g_{i,j}(p).
\end{align*}
Now we prove the case $j=0$. By \cref{502} we have
	\begin{equation*}
\mean\limits_{u\sim N(0,\sigma^2)} P_{0}(u)P_{2j}(u) e^{-\tfrac{u^2}{2}}=\frac{(-1)^{j} 4^{j+1} \Gamma\left(j+\frac{1}{2}\right)\left(\frac{p-1}{3p-2}\right)^{j+\tfrac{1}{2}}}{2\sqrt{2\pi}} =\frac{ (-1)\,\Gamma\left(j+\frac{1}{2}\right)}{2\sqrt{2\pi}\left(-\frac{3p-2}{4(p-1)}\right)^{j+1}} \, \sqrt{\frac{3p-2}{p-1}}.\label{204}
\end{equation*}
Moreover, by \cref{expectation_hermite2} (2) we have
\begin{align*}
\nonumber\mean\limits_{u\sim N(0,\sigma^2)} P_{-1}(u)P_{2j+1}(u) e^{-\tfrac{u^2}{2}} &= \frac{(-1)^{j+1} (2j+1)!}{2^{j}\,j!}\,\frac{(1-\sigma^2)^j\sigma^{2}}{\sqrt{1+\sigma^{2}} }\;F\left(-j,\tfrac{1}{2},\tfrac{3}{2},\tfrac{\sigma^4}{\sigma^4-1}\right)\\
\nonumber&= \frac{(-1)^{j+1} (2j+1)!}{2^{j}\,j!}\,\frac{(\tfrac{p-2}{2(p-1)})^j (\tfrac{p}{2(p-1)})}{\sqrt{\tfrac{3p-2}{2(p-1)}} }\;F\left(-j,\tfrac{1}{2},\tfrac{3}{2},\tfrac{-p^2}{(3p-2)(p-2)}\right)\\
&= \frac{(-1)^{j+1} (2j+1)!}{2^{2j}\,\sqrt{2}\,j!}\,\frac{(p-2)^jp\,\sqrt{\frac{3p-2}{p-1}}}{(p-1)^{j}(3p-2)}  \;F\left(-j,\tfrac{1}{2},\tfrac{3}{2},\tfrac{-p^2}{(3p-2)(p-2)}\right).
\end{align*}
Combining the previous two equations we see that $\mean_{u\sim N(0,\sigma^2)}e^\frac{-u^2}{2}\det\begin{bmatrix} P_{0}(u) & P_{2j+1}(u) \\ P_{-1}(u) & P_{2j}(u) \end{bmatrix}$ equals
\begin{align*}
\sqrt{\frac{3p-2}{p-1}}\left[\frac{(-1)^{j} (2j+1)!}{2^{2j}\,\sqrt{2}\,j!}\,\frac{(p-2)^jp}{(p-1)^{j}(3p-2)} \;F\left(-j,\tfrac{1}{2},\tfrac{3}{2},\tfrac{-p^2}{(3p-2)(p-2)}\right) - \frac{ \Gamma\left(j+\frac{1}{2}\right)}{2\sqrt{2\pi}\left(-\frac{3p-2}{4(p-1)}\right)^{j+1}} \right],
\end{align*}
which is equal to $\tfrac{1}{\sqrt{2\pi}}\sqrt{\tfrac{3p-2}{p-1}}\,g_{i,0}(p)$.
This finishes the proof.\qed
%%%%%%%%%%%%%%%%%%%%%%%%%%%%%%%%%%%%%%%%%%%%%%%%%%%%%%%%%%%%%%%%%%%%%%%
%%%%%%%%%%%%%%%%%%%%%%%%%%%%%%%%%%%%%%%%%%%%%%%%%%%%%%%%%%%%%%%%%%%%%%%
%%%%%%%%%%%%%%%%%%%%%%%%%%%%%%%%%%%%%%%%%%%%%%%%%%%%%%%%%%%%%%%%%%%%%%%
%%%%%%%%%%%%%%%%%%%%%%%%%%%%%%%%%%%%%%%%%%%%%%%%%%%%%%%%%%%%%%%%%%%%%%%
%%%%%%%%%%%%%%%%%%%%%%%%%%%%%%%%%%%%%%%%%%%%%%%%%%%%%%%%%%%%%%%%%%%%%%%
\section{Computation of integrals}\label{sec:integrals}
This section is dedicated to the computation of the integrals that appear in the next sections. We start with a general lemma that we will later apply to equation \cref{7} and equation~\cref{602}.
\begin{lemma}\label{auxiliary_lemma}
Let $f:\HR^m\times\HR\to \HR$, $(x,u)\mapsto f(x,u)$ be a measurable function, such that $\int_{\HR^m\times \HR^k} |f(x,u)| \, \d x \d u <\infty$. Assume that $f(x,u)$ is invariant under any permutations of the entries of $x=(x_1,\ldots,x_m)$. Then
	\[\sum_{j=0}^{m} \binom{m}{j}\,\int_{\substack{x_1,\ldots, x_j \leq u\\u\leq x_{j+1},\ldots,x_{m}} \; } f(x_1,\ldots,x_m,u)\;\d x_1\ldots  \d x_m = \int_{x_1,\ldots, x_m\in \HR} f(x_1,\ldots,x_m,u)\;\d x_1\ldots  \d x_m\]
for all $u\in \HR$.
\end{lemma}
\begin{proof}
We prove the statement by induction. For $m=1$ we have
	\begin{equation*}
	\int_{x_1\leq u} f(x_1,u)\;\d x_1 + \int_{ u \leq x_1} f(x_1,u)\;\d x_1=\int_{x_1\in\HR} f(x_1,u)\;\d x_1.
	\end{equation*}
For $m>1$ we write
	\begin{equation*}
	g_j(u):=\int_{\stackrel{x_1,\ldots, x_j \leq u}{u\leq x_{j+1},\ldots,x_{m}} \; } f(x_1,\ldots,x_m,u)\;\d x_1\ldots  \d x_m, \quad j=0,\ldots,m
	\end{equation*}
Using $\binom{m}{j}=\binom{m-1}{j}+\binom{m-1}{j-1}$ \cite[6:5:3]{atlas} we have
	\begin{equation}\label{501}
	\sum_{j=0}^m \binom{m}{j} g_j(u) = \sum_{j=0}^{m-1} \binom{m-1}{j}  (g_j(u)+g_{j+1}(u)),
	\end{equation}
where
\begin{equation*}
g_j(u)+g_{j+1}(u) = \int_{\substack{x_1,\ldots, x_j \leq u\\u\leq x_{j+2},\ldots,x_{m}}} \; \int_{x_{j+1}=-\infty}^\infty\, f(x_1,\ldots,x_m,u) \d x_1\ldots  \d x_m.
\end{equation*}
By assumption $f$ is invariant under any permutations of the $x_i$. Thus, making a change of variables that interchanges $x_{j+1}$ and $x_m$ we see that
\begin{equation*}
g_j(u)+g_{j+1}(u) = \int_{x_{m}=-\infty}^\infty \left[\; \int_{\substack{x_1,\ldots, x_j \leq u\\u\leq x_{j+1},\ldots,x_{m-1}}} \;  f(x_1,\ldots,x_{m-1},x_m,u) \d x_1\ldots  \d x_{m-1}\right] \d x_m.
\end{equation*}
Plugging this into \cref{501}, and using Fubini's theorem to interchange summation and integration we see that $\sum_{j=0}^{m} \binom{m}{j}\, g_j(u)$ is equal to
	\[ \int_{x_{m}=-\infty}^\infty\left[\sum_{j=0}^{m-1} \binom{m-1}{j}  \int_{\substack{x_1,\ldots, x_j \leq u\\u\leq x_{j+1},\ldots,x_{m-1}}} \;  f(x_1,\ldots,x_{m-1},x_m,u) \d x_1\ldots  \d x_{m-1}\right]\d x_m.\]
We can now apply the induction hypothesis to the inner integral and conclude the proof.
\end{proof}
Recall from \cref{my_polynomials} the definition of the polynomials $P_k(x)$ and from \cref{G} the definition of $G_k(x)$. The following technical lemma is the key to prove \cref{prop1} and \cref{prop2} below.
%%%%%%%%%%%%%%%%%%%%%%%%%%%%%%%%%%%%%%%%%%%%%%%%%%%%%%%%%%%%%%%%%%%%%%%
\begin{lemma}\label{prop0}
Let $\cA$ denote the $(2m+1)\times (2m-2)$ matrix
	\begin{equation*}
\cA = \begin{bmatrix}  G_{i}(x_1) & P_{i}(x_1) & \ldots & G_{i}(x_{m-1}) & P_{i} (x_{m-1}) \end{bmatrix}_{0\leq i \leq 2m}.
\end{equation*}
For a subset $S\subset \{0,\ldots,2m\}$ with three elements let $\cA^{S}$ be the $2(m-1)\times 2(m-1)$-matrix that is obtained by removing from $\cA$ all the rows indexed by $S$. We write
	\begin{equation}\label{xi}
	\Xi(S):=\int_{x_1,\ldots,x_{m-1}\in\HR } \det(\cA^S)\; e^{-\sum\limits_{i=1}^{m-1} \tfrac{x_i^2}{2}}\;\d x_1\ldots  \d x_{m-1}.
	\end{equation}
Then, we have
\begin{enumerate}
\item
Let $S_0\subset \set{1,\ldots,2m}$ be a subset with $2$ elements and $S=S_0\cup \set{0}$. Then
\begin{align*}
\Xi(S)
=&\begin{cases}  (m-1)! \,2^{m-1}\,\det(\Gamma_1^{a,b}), & \text{if } S=\set{0,2a-1,2b} \text{ and } a\leq b\\
-(m-1)! \,2^{m-1}\,\det(\Gamma_1^{a,b}), & \text{if } S=\set{0, 2a-1,2b} \text{ and } a> b\\
0,&\text{if } S \text{ has any other form.}
\end{cases}
\end{align*}
\item Let $S_0\subset \set{0,\ldots,2m-1}$ be a subset with $2$ elements and $S=S_0\cup\set{2m}$. Then
\begin{align*}
\Xi(S)
=&\begin{cases}  (m-1)! \,2^{m-1}\,\det(\Gamma_2^{a,b}), & \text{if } S=\set{2a,2b+1,2m} \text{ and } a\leq b\\
-(m-1)! \,2^{m-1}\det(\Gamma_2^{a,b}), & \text{if } S=\set{2a,2b+1,2m} \text{ and } a> b\\
0,&\text{if } S \text{ has any other form.}
\end{cases}
\end{align*}
\end{enumerate}
Here, $\Gamma_1^{a,b}$ and $\Gamma_2^{a,b}$ are the matrices from \cref{gamma_matrices}.
\end{lemma}
%%%%%%%%%%%%%%%%%%%%%%%%%%%%%%%%%%%%%%%%%%%%%%%%%%%%%%%%%%%%%%%%%%%%%%%
\begin{proof}
We first prove (1). Fix $S_0\subset \set{1,\ldots,2m}$ and $S=S_0\cup\set{0}$. Then
	\begin{equation}\label{99}
	\cA^S= \begin{bmatrix}  G_{i}(x_1) & P_{i}(x_1) & \ldots & G_{i}(x_{m}) & P_{i}(x_{m-1}) \end{bmatrix}_{1\leq i \leq 2m, i\not \in  S_0}.
	\end{equation}
To ease notation put
	\begin{equation}\label{m_dash}
	\mu:=m-1.
	\end{equation}
Furthermore, let us denote the elements in $\set{1,\ldots,2m}\backslash S_0$ in ascending order by $s_1<\ldots < s_{2\mu}$ and let $\kS_{2\mu}$ denote the group of permutations on $\set{1,\ldots,2\mu}$. Expanding the determinant of $\cA^S$ yields
	\begin{equation}\label{det}
	\det(\cA^S) = \sum_{\pi \in \kS_{2\mu}} \mathrm{sgn}(\pi) \, \prod_{i=1}^{\mu} G_{s_{\pi(2i-1)}}(x_i) \, P_{s_{\pi(2i)}}(x_i).
	\end{equation}
Recall from \cref{inner_product} the definition of $\langle \_,\_\rangle$. Plugging \cref{det} into \cref{xi} and integrating over all the~$x_i$ we see that
\begin{equation}\label{a3}
\Xi(S)=\sum_{\pi \in \kS_{2\mu}} \mathrm{sgn}(\pi) \, \prod_{i=1}^{\mu} \langle G_{s_{\pi(2i-1)}}(x), P_{s_{\pi(2i)}}(x) \rangle.
\end{equation}
From \cref{lemma2} we know that $\langle G_{k}(x), P_{\ell}(x) \rangle=0$ whenever $k+\ell$ is even and $k\neq \ell$. This already proves that $\Xi(S)=0$, if $S$ is not of the form $S=\set{0,2a-1,2b}$, because in this case we can't have a partition of $\set{1,\ldots,2m}\backslash S$ into pairs of numbers where one number is even and the other is odd. If, on the other hand, $S=\set{0,2a-1,2b}$ does contain one odd and two even elements, in~\cref{a3} we may as well sum over the subset
	\[\kO_{2\mu}:= \cset{\pi\in\kS_{2\mu}}{\forall i\in \set{1,\ldots,\mu}: s_{\pi(2i-1)} + s_{\pi(2i)} \text{ is odd} }.\]
Let $\kT\subset\kS_{2\mu}$ be the subgroup generated by the set of transpositions $\set{(1\; 2), (3\; 4), \ldots , ((2\mu-1)\; 2\mu)}$. We define an equivalence relation on $\kO_{2\mu}$ via:
\[\forall \pi,\sigma\in \kO_{2\mu}: \quad \pi\sim\sigma \;:\Leftrightarrow \; \exists \tau \in \kT: \pi = \sigma \tau.\]
Note that the multiplication with $\tau$ from the right is crucial here. Let $\kC:=\kO_{2\mu} /\sim$ denote the set of equivalence classes of $\kT$ in $\kO_{2\mu}$. A set of representatives for $\kC$ is
	\[\kR:=\cset{\pi \in \kS_{2\mu}}{\forall i \in \set{1,\ldots,\mu}: s_{\pi(2i-1)} \text{ is odd and } s_{\pi(2i)}\text{ is even}}.\]
Making a partition of $\kO_{2\mu}$ into the equivalence classes of $\sim$ in \cref{a3} we get
\begin{equation*}
\Xi(S)=\sum_{\pi\in \kR}\sum_{\tau\in \kT} \mathrm{sgn}(\pi\circ\tau)\, \prod_{i=1}^\mu \langle G_{s_{\pi\circ\tau(2i-1)}}(x), P_{s_{\pi\circ\tau(2i)}}(x) \rangle.
\end{equation*}
For a fixed $\pi\in\kR$ and all $\tau \in \kT$, by \cref{orthogonal_rel} (1) we have
\[\prod_{i=1}^\mu \langle G_{s_{\pi\circ\tau(2i-1)}}(x), P_{s_{\pi\circ\tau(2i)}}(x) \rangle = \mathrm{sgn}(\tau) \prod_{i=1}^\mu \langle G_{s_{\pi(2i-1)}}(x), P_{s_{\pi(2i)}}(x) \rangle\]
so that
\begin{equation*}
\Xi(S)=2^\mu \sum_{\pi\in \kR}\mathrm{sgn}(\pi)\, \prod_{i=1}^\mu \langle G_{s_{\pi(2i-1)}}(x), P_{s_{\pi(2i)}}(x) \rangle.
\end{equation*}
Let us investigate $\kR$ further. We denote the group of permutation on $\set{1,\ldots,\mu}$ by $\kS_\mu$. The group $\kS_\mu\times\kS_\mu$ acts transitively and faithful on $\kR$ via
	\[\forall i: \left((\sigma_1,\sigma_2).\pi\right)(2i-1) := \pi(2\sigma_1(i)-1)\quad\text{and}\quad \left((\sigma_1,\sigma_2).\pi\right)(2i) := \pi(2\sigma_2(i))\]
This shows that that we have a bijection $\kS_\mu\times\kS_\mu\to \kR,\; (\sigma_1,\sigma_2) \mapsto (\sigma_1,\sigma_2).\pi_\star$ where $\pi_\star \in \kR$ is fixed. Moreover, for all $(\sigma_1,\sigma_2)\in\kS_\mu\times \kS_\mu$ we have $\mathrm{sign}((\sigma_1,\sigma_2).\pi_\star)=\mathrm{sgn}(\sigma_1)\mathrm{sgn}(\sigma_2)\mathrm{sign}(\pi_\star)$.

Let us denote $2k_i-1=s_{\pi_\star(2i-1)}$ and $2\ell_i=s_{\pi_\star(2i)}$. We choose $\pi_\star$ uniquely by requiring $k_1<k_2<\ldots <k_\mu$ and $\ell_1<\ell_2<\ldots<\ell_\mu$. By doing so we get
\begin{align}
\label{x}\Xi(S)=&2^\mu \mathrm{sgn}(\pi_\star)\sum_{(\sigma_1,\sigma_2)\in\kS_\mu\times\kS_\mu}\mathrm{sgn}(\sigma_1)\mathrm{sgn}(\sigma_2)\, \prod_{i=1}^\mu \langle G_{2k_{\sigma_1(i)}-1}(x), P_{2\ell_{\sigma_2(i)}}(x) \rangle\\
=&2^\mu \mu!\, \mathrm{sgn}(\pi_\star)\sum_{\sigma\in\kS_\mu}\mathrm{sgn}(\sigma)\, \prod_{i=1}^\mu \langle G_{2k_{\sigma(i)}-1}(x), P_{2\ell_{i}}(x) \rangle.\nonumber\\
=&2^\mu \mu! \,\mathrm{sgn}(\pi_\star)\sum_{\sigma\in\kS_\mu}\mathrm{sgn}(\sigma)\,\prod_{i=1}^\mu (-1)^{k_{\sigma(i)}+\ell_i} \, \Gamma\left(k_{\sigma(i)}+\ell_i-\tfrac{1}{2}\right) \nonumber
\end{align}
the last line by \cref{orthogonal_rel} (2). By construction we have $\bigcup_{i=1}^\mu\set{2k_i-1,2\ell_i}=\set{1,\ldots,2m}\backslash S_0$, so that
	\[\set{k_1,\ldots, k_\mu} = \set{1,\ldots,m}\backslash \set{a}, \text{ and } \set{\ell_1,\ldots, \ell_\mu} = \set{1,\ldots,m}\backslash \set{b}.\]
Hence, for all $\sigma\in\kS_\mu$ we have
	\begin{equation}\label{x2}
	\prod_{i=1}^\mu (-1)^{k_{\sigma(i)}+\ell_i} =(-1)^{m(m+1)-a-b} = (-1)^{a+b}.
	\end{equation}
and, furthermore,
	\begin{equation}
	\label{x4}
	\mathrm{sgn}(\pi_\star)=\begin{cases} (-1)^{a+b},& \text{if } a\leq b \\ (-1)^{a+b-1}, &\text{if } a>b\end{cases}.
	\end{equation}
Moreover,
	\begin{equation}\label{x3}
	\sum_{\sigma\in\kS_\mu}\mathrm{sgn}(\sigma)\,\prod_{i=1}^\mu\Gamma\left(k_{\sigma(i)}+\ell_i-\tfrac{1}{2}\right) = \det\left(\left[\Gamma\left(k+\ell-\tfrac{1}{2}\right)\right]_{\substack{1\leq k\leq m, k\neq a,\\1\leq \ell\leq m, \ell\neq b.}}\right) = \det(\Gamma_1^{a,b}),
	\end{equation}
Putting together \cref{x}, \cref{x4}, \cref{x2} and \cref{x3} proves (1).

We now prove (2). Fix $S_0\subset\set{0,\ldots,2m-1}$ and let $S=S_0\cup\{2m\}$. Similar to \cref{99} we have
	\begin{equation*}
	\cA^S= \begin{bmatrix}  G_{i}(x_1) & P_{i}(x_1) & \ldots & G_{i}(x_{m}) & P_{i}(x_{m}) \end{bmatrix}_{0\leq i \leq 2m-1, i\not\in S_0}.
	\end{equation*}
Put $\widetilde{G}_i(x):=G_{i-1}(x)$ and $\widetilde{P}_i(x):=P_{i-1}(x)$, so that
	\begin{equation*}
	\cA^S= \begin{bmatrix}  \widetilde{G}_{i}(x_1) & \widetilde{P}_{i}(x_1) & \ldots & \widetilde{G}_{i}(x_{m}) & \widetilde{P}_{i}(x_{m}) \end{bmatrix}_{1\leq i \leq 2m, i\not\in \widetilde{S}_0},
	\end{equation*}
where $\widetilde{S}_0$ is the set that is obtained from $S_0$ by adding $1$ to both elements of $S_0$ and $\widetilde{S}=\widetilde{S}_0\cup\{0\}$.
We now proceed as in the proof of (1) until \cref{x}, and conclude that
	\[\Xi(S)=\begin{cases}  (m-1)! \,2^{m-1}\,\det(\,\widetilde{\Gamma}_2^{a',b'}\,), & \text{if } \widetilde{S}=\set{0, 2a'-1,2b'} \text{ and } a\leq b\\
-(m-1)! \,2^{m-1}\,\det(\widetilde{\Gamma}_2^{a',b'}), & \text{if } \widetilde{S}=\set{0, 2a'-1,2b'} \text{ and } a> b\\
0,&\text{if } \widetilde{S} \text{ has any other form.}
\end{cases},\]
where
\[\widetilde{\Gamma}_2^{a',b'}:=\left[\Gamma\left(k+\ell-\tfrac{3}{2}\right)\right]_{\substack{1\leq k\leq m, k\neq a'\\1\leq \ell\leq m, \ell\neq b'}}= \Gamma_2^{a'-1,b'-1}.\]
If $\widetilde{S}=\set{0,2a'-1,2b'}$ then, by definition, $S=\set{2a,2b+1,2m}$, where $a=a'-1$ and $b=b'-1$. Hence,
	\begin{align*}
	&\Xi(s)
	=\begin{cases}  (m-1)! \,2^{m-1}\,\det({\Gamma}_2^{a,b}), & \text{if } {S}=\set{2a,2b+1,2m} \text{ and } a\leq b\\
-(m-1)! \,2^{m-1}\,\det({\Gamma}_2^{a,b}), & \text{if } {S}=\set{2a,2b+1,2m} \text{ and } a> b\\
0,&\text{if } {S} \text{ has any other form.}
\end{cases},\end{align*}
This finishes the proof.
\end{proof}
\cref{prop1} and \cref{prop2} below become important in \cref{sec_n_even} and \cref{sec_n_odd}, respectively.
\begin{prop}\label{prop1}
Recall from \cref{sec:orth_rel} the definition of $P_k(x)$ and~$G_k(x)$. Let $\cM$ denote the matrix
	\[\cM:=\begin{bmatrix} P_{i}(u) &  \begin{bmatrix}G_{i}(x_j) & P_{i}(x_j)\end{bmatrix}_{{j=1,\ldots,m-1 }}   &G_{i}(u)& G_i(\infty)\end{bmatrix}_{i=0,\ldots,2m}\]
We have
\begin{align*}&\int_{x_1,\ldots,x_{m-1} \in\HR} \det(\cM)\,  e^{-\sum\limits_{i=1}^{m-1} \tfrac{x_i^2}{2}}\,\d x_1\ldots \d x_{m-1}\\
= \;&\sqrt{2\pi} (m-1)! \,2^{\,m-1}\,e^{-\tfrac{u^2}{2}}\,\sum_{1\leq i,j \leq m} \det(\Gamma^{i,j}_1) \; \det\begin{bmatrix}  P_{2j}(u)&P_{2i-1}(u)  \\ P_{2j-1}(u)&P_{2i-2}(u)   \end{bmatrix}.
\end{align*}
where the matrix $\Gamma^{i,j}_1$ is defined as in \cref{gamma_matrices}.
\end{prop}
\begin{proof}
Let us denote the quantity that we want to compute by $\Theta$:
	\[\Theta:=\int_{x_1,\ldots,x_{m-1} \in\HR} \det (\cM)  e^{-\sum\limits_{i=1}^{m-1} \tfrac{x_i^2}{2}} \d x_1\ldots \d x_{m-1}.\]
A permutation with negative sign of the columns of $\cM$ yields,
\begin{equation}
\det(\cM)=- \det \begin{bmatrix}    G_i(\infty)&G_{i}(u) & P_{i}(u) & \begin{bmatrix}G_{i}(x_j) & P_{i}(x_j)\end{bmatrix}_{{j=1,\ldots,m-1 }} \end{bmatrix}_{i=0,\ldots,2m}\label{20}
\end{equation}
By \cref{lemma2} we have $G_{i}(\infty)=0$ for $i\geq 1$ and $G_0(\infty)=\sqrt{2\pi}$. Expanding the determinant in \cref{20} with Laplace expansion we get
	\[\det(\cM)=-\sqrt{2\pi} \sum_{1\leq k<\ell \leq 2m} (-1)^{k+\ell-1} (G_k(u)P_\ell(u)-G_\ell(u)P_k(u))\, \det(\cA^{k,\ell}),\]
where
$$\cA^{k,\ell}:=\begin{bmatrix}G_{i}(x_j) & P_{i}(x_j)\end{bmatrix}_{\substack{1\leq i\leq 2m, i\not\in\set{k,\ell}\\ 1\leq j\leq m-1\quad\quad\;}}.$$
In the notation of \cref{prop0} we have $\cA^{k,\ell}=\cA^{\set{0,k,\ell}}$ and
	$$\Xi(\set{0,k,\ell}) = \int_{x_1,\ldots,x_{m-1}\in\HR} \det(\cA^{k,\ell}) e^{-\sum\limits_{i=1}^{m-1} \tfrac{x_i^2}{2}} \d x_1 \ldots \d x_{m-1},$$
so that
\begin{equation}\label{a31.2}\Theta = \sqrt{2\pi} \sum_{1\leq k<\ell  \leq 2m} (-1)^{k+\ell}(G_k(u)P_\ell(u)-G_\ell(u)P_k(u)) \, \Xi(\set{0,k,\ell}).\end{equation}
Applying the \cref{prop0} yields
	\begin{equation*}
	  \Xi(\set{0,k,\ell})
	 =
\begin{cases}
(m-1)! 2^{m-1}  \det(\Gamma_1^{i,j}), &\text{ if } \set{k,\ell}=\set{2i-1,2j}, i\leq j.\\
 -(m-1)! 2^{m-1} \det(\Gamma_1^{i,j}), &\text{ if } \set{k,\ell}=\set{2i-1,2j}, i> j.\\
 0,& \text{ else.}\end{cases}
 \end{equation*}
When we want to plug this into \cref{a31.2} we have to incorporate that
\[\begin{cases} \text{If } k=2i-1,\ell=2j \text{ and } k<\ell, \text{ then } i\leq j.\\
\text{If } k=2j,\ell=2i-1 \text{ and } k<\ell, \text{ then } i>j.
\end{cases}\]
From this we get
\begin{align*}
	\Theta=&(-1)\,\sqrt{2\pi} (m-1)! \,2^{\,m-1}\,e^{-\tfrac{u^2}{2}}\,\Big[\sum_{1\leq i\leq j \leq m} \det(\Gamma^{i,j}_1) \; (G_{2i-1}(u)P_{2j}(u)-G_{2j}(u)P_{2i-1}(u))\\
	&\hspace{5cm}-\sum_{1\leq j < i \leq m} \det(\Gamma^{i,j}_1) \; (G_{2j}(u)P_{2i-1}(u)-G_{2i-1}(u)P_{2j}(u))\Big]
	\end{align*}
By \cref{lemma2} we have $G_k(u)=-e^{-\tfrac{u^2}{2}}  P_{k-1}(u)$, $k\geq 1$, which we plug into the upper expression:
	\begin{align*}
	\Theta=&\sqrt{2\pi} (m-1)! \,2^{\,m-1}\,e^{-\tfrac{u^2}{2}}\,\Big[\sum_{1\leq i\leq j \leq m} \det(\Gamma^{i,j}_1) \; (P_{2i-2}(u)P_{2j}(u)-P_{2j-1}(u)P_{2i-1}(u))\\
	&\hspace{4cm}-\sum_{1\leq j < i \leq m} \det(\Gamma^{i,j}_1) \; (P_{2j-1}(u)P_{2i-1}(u)-P_{2i-2}(u)P_{2j}(u))\Big]\\
=&	\sqrt{2\pi} (m-1)! \,2^{\,m-1}\,e^{-\tfrac{u^2}{2}}\,\sum_{1\leq i,j \leq m} \det(\Gamma^{i,j}_1) \; (P_{2i-2}(u)P_{2j}(u)-P_{2j-1}(u)P_{2i-1}(u))\\
=&	\sqrt{2\pi} (m-1)! \,2^{\,m-1}\,e^{-\tfrac{u^2}{2}}\,\sum_{1\leq i,j \leq m} \det(\Gamma^{i,j}_1) \; \det\begin{bmatrix}  P_{2j}(u)&P_{2i-1}(u)  \\ P_{2j-1}(u)&P_{2i-2}(u)   \end{bmatrix}
	\end{align*}
This finishes the proof.
\end{proof}
\begin{prop}\label{prop2}
Denote the matrix
$$\cM= \begin{bmatrix} P_{i}(u) & \begin{bmatrix}G_{i}(x_j) & P_{i}(x_j)\end{bmatrix}_{j=1,\ldots, m-1} & G_i(u)\end{bmatrix}_{i=0,\ldots,2m-1}.$$ Then
	\begin{align*}
	&\int_{x_1,\ldots,x_{m-1}\in\HR} \det(\cM) \, e^{-\sum\limits\limits_{i=1}^{m-1} \tfrac{x_i^2}{2}}\, \d x_1 \ldots \d x_{m-1}\\
	= &(m-1)! \,2^{m-1}\,e^{-\tfrac{u^2}{2}}\sum_{0\leq i,j\leq m-1}  \det(\Gamma_2^{i,j})\, \det\begin{bmatrix} P_{2i}(u) & P_{2j+1}(u) \\ P_{2i-1}(u) & P_{2j}(u) \end{bmatrix},
	\end{align*}
where the matrix $\Gamma_2^{i,j}$ is defined as in \cref{gamma_matrices}.
\end{prop}
\begin{proof}
The proof works similar as the the proof for \cref{prop1}: Again, we denote by $\Theta$ the quantity that we want to compute:
\[\Theta:=\int_{x_1,\ldots,x_m\in\HR} \det(\cM) \, e^{-\sum\limits_{i=1}^m \tfrac{x_i^2}{2}}\, \d x_1 \ldots \d x_m.\]
We have
	\[\det(\cM)=-\det \begin{bmatrix}  G_i(u)& P_{i}(u) &\begin{bmatrix}G_{i}(x_j) & P_{i}(x_j)\end{bmatrix}_{j=1,\ldots, m} \end{bmatrix}_{i=0,\ldots,2m-1}.\]
Expanding the determinant with Laplace expansion we get
	\[\det(\cM)= -\sum_{0\leq k < \ell \leq 2m-1} (-1)^{k+\ell-1} (G_k(u)P_\ell(u) - G_\ell(u)P_k(u)) \det(\cA^{k,\ell}),\]
where
	\[\cA^{k,\ell}=\begin{bmatrix}G_{i}(x_1) & P_{i}(x_1)& \ldots &G_{i}(x_m) & P_{j}(x_m)\end{bmatrix}_{i=0,\ldots,2m-1, i\not\in\set{k,\ell}}\]
In the notation \cref{prop0} of we have
	$$\Xi(\set{k,\ell,2m}) = \int_{x_1,\ldots,x_{m-1}\in\HR} \det(\cA^{k,\ell}) e^{-\sum\limits_{i=1}^{m-1} \tfrac{x_i^2}{2}} \d x_1 \ldots \d x_{m-1}.$$
Hence,
\begin{equation}\label{101}
\Theta=\sum_{0\leq k < \ell \leq 2m+1} (-1)^{k+\ell} (G_k(u)P_\ell(u) - G_\ell(u)P_k(u)) \,\Xi(\set{k,\ell,2m}).
\end{equation}
By \cref{prop0} (2) we have
	\begin{align*}
\Xi({k,\ell,2m})
=\begin{cases}  (m-1)! \,2^{m-1}\,\det(\Gamma_2^{i,j}), & \text{if } \set{k,\ell}=\set{2j,2i+1} \text{ and } j\leq i\\
-(m-1)! \,2^{m-1}\,\det(\Gamma_2^{i,j}), & \text{if } \set{k,\ell}=\set{2j,2i+1} \text{ and } j> i\\
0,&\text{else.}\end{cases}
 \end{align*}
When pluggin this into \cref{101} we must take into account that
\[\begin{cases} \text{If } k=2j,\ell=2i+1 \text{ and } k<\ell, \text{ then } j\leq i.\\
\text{If } k=2i+1,\ell=2j \text{ and } k<\ell, \text{ then } j > i.
\end{cases}\]
This yields
\begin{align*}
\Theta&=(m-1)! \,2^{m-1}\, \big[\sum_{0\leq i<j\leq m-1}  \det(\Gamma_2^{i,j})\, (G_{2i+1}(u)P_{2j}(u) - G_{2j}(u)P_{2i+1}(u) ) \\
 &\hspace{2cm}- \sum_{0\leq j\leq i\leq m-1}  \det(\Gamma_2^{i,j})\, (G_{2j}(u)P_{2i+1}(u) - G_{2i+1}(u)P_{2j}(u))\big]\\
 &=m! \,2^{m-1}\,\sum_{0\leq i,j\leq m-1}  \det(\Gamma_2^{i,j})\, (P_{2j}(u)G_{2i+1}(u) - P_{2i+1}(u)G_{2j}(u))
\end{align*}
Using from \cref{lemma2} that $G_k(u)=-e^{-\tfrac{u^2}{2}}  P_{k-1}(u)$, we finally obtain
	\begin{align*}
	\Theta=(m-1)! \,2^{m-1}\,e^{-\tfrac{u^2}{2}}\,\sum_{0\leq i,j\leq m-1}  \det(\Gamma_2^{i,j})\, \det\begin{bmatrix}   P_{2i+1}(u)&P_{2j}(u) \\  P_{2i}(u) & P_{2j-1}(u) \end{bmatrix}.
	\end{align*}
This finishes the proof.
\end{proof}
%%%%%%%%%%%%%%%%%%%%%%%%%%%%%%%%%%%%%%%%%%%%%%%%%%%%%%%%%%%%%%%%%%%%%%%
%%%%%%%%%%%%%%%%%%%%%%%%%%%%%%%%%%%%%%%%%%%%%%%%%%%%%%%%%%%%%%%%%%%%%%%
%%%%%%%%%%%%%%%%%%%%%%%%%%%%%%%%%%%%%%%%%%%%%%%%%%%%%%%%%%%%%%%%%%%%%%%
%%%%%%%%%%%%%%%%%%%%%%%%%%%%%%%%%%%%%%%%%%%%%%%%%%%%%%%%%%%%%%%%%%%%%%%
%%%%%%%%%%%%%%%%%%%%%%%%%%%%%%%%%%%%%%%%%%%%%%%%%%%%%%%%%%%%%%%%%%%%%%%
%%%%%%%%%%%%%%%%%%%%%%%%%%%%%%%%%%%%%%%%%%%%%%%%%%%%%%%%%%%%%%%%%%%%%%%
%%%%%%%%%%%%%%%%%%%%%%%%%%%%%%%%%%%%%%%%%%%%%%%%%%%%%%%%%%%%%%%%%%%%%%%
%%%%%%%%%%%%%%%%%%%%%%%%%%%%%%%%%%%%%%%%%%%%%%%%%%%%%%%%%%%%%%%%%%%%%%%
%%%%%%%%%%%%%%%%%%%%%%%%%%%%%%%%%%%%%%%%%%%%%%%%%%%%%%%%%%%%%%%%%%%%%%%
\section{Proof of Theorem \ref{thm}}\label{sec:111}
The proof of Theorem \ref{thm} is an adaption of the computations in \cite[Sec. 22]{mehta}. Recall from~\cref{def_I_J} that we have put
\begin{equation*}
\cI_n(u):=\mean\limits_{A\sim \mathrm{GOE}(n)} \lvert \det(A-uI_n)\rvert,\quad\text{and}\quad \cJ_n(u):=\mean\limits_{A\sim \mathrm{GOE}(n)}\det(A-uI_n),
\end{equation*}
The proof of \cref{thm} is based on the idea to decompose $\cI_n(u)=(\cI_n(u)+\cJ_n(u)) - \cJ_n(u)$ and then to compute the two summands $\cI_n(u)+\cJ_n(u)$ and $\cJ_n(u)$ separately. By definition of the Gaussian Orthogonal Ensemble we have
\begin{equation*}
\cI_n(u)
=\frac{1}{\sqrt{2}^{\,n} \sqrt{\pi}^{\,n(n+1)/2}}\;  \int_{A \in S^2(\HR^n)}\, \lvert\det(A-uI_{n})\rvert \; e^{-\tfrac{1}{2}\,\Vert A\Vert^2_F} \,\d A.
\end{equation*}
By \cite[Theorem 3.2.17]{muirhead}, the density of the (ordered) eigenvalues $\lambda_1\leq\ldots\leq \lambda_n$ of $A\sim \mathrm{GOE}(n)$ is given by
	\[\frac{\sqrt{\pi}^{\tfrac{n(n+1)}{2}}}{ \prod_{i=1}^n\Gamma\left(\tfrac{i}{2}\right)}\; \Delta(\lambda) \,e^{-\sum\limits_{i=1}^{n} \tfrac{\lambda_i^2}{2} }  \,\mathbf{1}_{\set{\lambda_1\leq \ldots\leq \lambda_n}},\]
where $\Delta(\lambda):= \prod_{1\leq i < j \leq n} (\lambda_j - \lambda_i)$ and $\mathbf{1}_{\set{\lambda_1\leq \ldots\leq \lambda_n}}$ is the characteristic function of the set $\set{\lambda_1\leq \ldots\leq \lambda_n}$. This implies
	\begin{equation}\label{I_n2}
	\cI_n(u) =  \frac{1}{\sqrt{2}^{\,n} \prod_{i=1}^n\Gamma\left(\tfrac{i}{2}\right)}\; \int_{\lambda_1\leq ...\leq \lambda_{n}} \Delta(\lambda) \,e^{-\sum\limits_{i=1}^{n} \tfrac{\lambda_i^2}{2} } \;\prod_{i=1}^n \lvert\lambda_i -u\rvert\;\d \lambda_1 \ldots \d \lambda_{n}.
	\end{equation}
Similiarly,
	\begin{equation}\label{J_n}
	\cJ_n(u)=  \frac{1}{\sqrt{2}^{\,n} \; \prod_{i=1}^n\Gamma\left(\tfrac{i}{2}\right)}\;  \int_{\lambda_1\leq ...\leq \lambda_{n}}\, \Delta(\lambda) \,e^{-\sum\limits_{i=1}^{n} \tfrac{\lambda_i^2}{2} }  \;\prod_{i=1}^n (\lambda_i -u )\;\d \lambda_1 \ldots \d \lambda_{n}.
	\end{equation}
In the remainder of the section we put $\lambda_0:=-\infty$ and
\begin{equation*}
c_n:= \frac{1}{\sqrt{2}^{\,n} \; \prod_{i=1}^n\Gamma\left(\tfrac{i}{2}\right)}.
\end{equation*}
Then, we can write \cref{I_n2} as
\[\cI_n(u)= c_n\,\sum_{j=0}^{n} (-1)^{j} \int_{\substack{\lambda_0\leq \lambda_1\leq ...\leq \lambda_{j} \leq u\\u\leq \lambda_{j+1}\leq \ldots \leq \lambda_n}}\,\Delta(\lambda)\, e^{-\sum\limits_{i=1}^{n} \tfrac{\lambda_i^2}{2} } \;\prod_{i=1}^n ( \lambda_i -u )\; \d \lambda_1 \ldots \d \lambda_{n}\]
and \cref{J_n} as
\[\cJ_n(u)=  c_n\,\sum_{j=0}^{n}\;\,\int_{\substack{\lambda_0\leq \lambda_1\leq ...\leq \lambda_{j} \leq u\\u\leq \lambda_{j+1}\leq \ldots \leq \lambda_n}}\, \Delta(\lambda)\, e^{-\sum\limits_{i=1}^{n} \tfrac{\lambda_i^2}{2} } \;\prod_{i=1}^n (\lambda_i -u)\; \d \lambda_1 \ldots \d \lambda_{n}.\]
Hence,
\begin{equation}
\cI_n(u)+\cJ_n(u)= 2c_n\;\sum_{j=0}^{\lfloor \tfrac{n}{2}\rfloor} \; \int_{\substack{\lambda_0\leq \lambda_1\leq ...\leq \lambda_{2j} \leq u\\u\leq \lambda_{2j+1}\leq \ldots \leq \lambda_n}} \Delta(\lambda) \, e^{-\sum\limits_{i=1}^{n} \tfrac{\lambda_i^2}{2} } \;\prod_{i=1}^n (\lambda_i -u)\; \d \lambda_1 \ldots \d \lambda_{n}\label{aaa}
\end{equation}
We write $ \Delta(\lambda)\prod_{i=1}^n (\lambda_i-u)$ as a Vandermonde determinant:
	\begin{equation*} \Delta(\lambda)\prod_{i=1}^n (\lambda_i-u)  =  \prod_{1\leq i<j\leq n} (\lambda_j  - \lambda_i) \; \prod_{i=1}^n (\lambda_i-u)= \det \begin{bmatrix}u^k & \lambda_1^{k} & \ldots & \lambda_{n}^{k} \end{bmatrix}_{k=0,\ldots, n}.
	\end{equation*}
Since we may add arbitrary multiple of rows to other rows of a matrix without changing its determinant, we have
\begin{equation}\label{b}
	 \Delta(\lambda)\prod_{i=1}^n (\lambda_i-u) = \det \begin{bmatrix} P_{k}(u) & P_{k}(\lambda_1) & \ldots & P_{k}(\lambda_{n})\end{bmatrix}_{k=0,\ldots, n},
\end{equation}
where the $P_{k}(x), k=0,1,\ldots,n,$ are the Hermite polynomials from \cref{my_polynomials}. Plugging this into~\cref{aaa} shows that $\cI_n(u)+\cJ_n(u)$ is equal to
\begin{align}
&\label{a}
2c_n\;\sum_{j=0}^{\lfloor \tfrac{n}{2}\rfloor} \; \int_{\substack{\lambda_0\leq \lambda_1\leq ...\leq \lambda_{2j} \leq u\\u\leq \lambda_{2j+1}\leq \ldots \leq \lambda_n}} \det \begin{bmatrix} P_{k}(u) & P_{k}(\lambda_1) & \ldots & P_{k}(\lambda_{n})\end{bmatrix}_{k=0,\ldots, n} e^{-\sum\limits_{i=1}^{n} \tfrac{\lambda_i^2}{2} } \, \d \lambda_1 \ldots \d \lambda_{n}
\end{align}
We now distinguish the cases $n$ even and $n$ odd.

%%%%%%%%%%%%%%%%%%%%%%%%%%%%%%%%%%%%%%%%%%%%%%%%%%%%%%%%%%%%%%%%%%%%%%%
%%%%%%%%%%%%%%%%%%%%%%%%%%%%%%%%%%%%%%%%%%%%%%%%%%%%%%%%%%%%%%%%%%%%%%%
%%%%%%%%%%%%%%%%%%%%%%%%%%%%%%%%%%%%%%%%%%%%%%%%%%%%%%%%%%%%%%%%%%%%%%%
\subsection{The case when $n=2m$ is even} \label{sec_n_even}
Recall from \cref{G} that we have put
\[G_{k}(x)=\int_{-\infty}^x P_{k}(y)\, e^{-\tfrac{y^2}{2}} \;\d y\]
Observe that each $\lambda_i$ appears in exactly one column on the right hand side of \cref{b}. Integrating over $\lambda_1,\lambda_3,\lambda_5,\ldots$ in \cref{a} therefore yields
	\begin{equation}
\cI_{2m}(u)+\cJ_{2m}(u)= 2c_{2m}\sum_{j=0}^{m} \;  \int_{\substack{ \lambda_2\leq \lambda_4\leq ...\leq \lambda_{2j} \leq u\\u\leq \lambda_{2j+2}\leq \ldots \leq \lambda_{2m}}}\det(\cX_j) \; e^{-\sum\limits_{i=1}^{m} \frac{\lambda_{2i}^2}{2} }\,\d \lambda_2  \ldots \d \lambda_{2m}
\end{equation}
where $\cX_j$ is the matrix
\begin{align*}
\cX_j=&\Big[ P_{k}(u)\;  \begin{bmatrix}G_{k}(\lambda_{2i}) - G_k(\lambda_{2i-2}) & P_{k}(\lambda_{2i})\end{bmatrix}_{{i=1,\ldots j}}\;  \begin{bmatrix}G_{k}(\lambda_{2j+2}) - G_k(u) & P_{k}(\lambda_{2j+2})\end{bmatrix}\,\ldots  \\
&\hspace{1cm} \ldots\, \begin{bmatrix}G_{k}(\lambda_{2i}) - G_{k}(\lambda_{2i-2}) & P_{k}(\lambda_{2i})\end{bmatrix}_{{i=j+2,\ldots, m}}\Big]_{k=0,\ldots,2m}\in\HR^{(2m+1)\times (2m+1)}.
\end{align*}
Adding the first column of $\cX_j$ to its third column, and the result to the fifth column and so on, does not change the value of the determinant. Hence, $\det(\cX_j)=\det(\cY_j)$, where
\begin{equation*}
\cY_j:=\begin{bmatrix} P_{k}(u)& \begin{bmatrix}G_{k}(\lambda_{2i}) & P_{k}(\lambda_{2i})\end{bmatrix}_{{i=1,\ldots j}} &  \begin{bmatrix}G_{k}(\lambda_{2i}) - G_{k}(u) & P_{k}(\lambda_{2i})\end{bmatrix}_{{i=j+1,\ldots, m}}\end{bmatrix}_{k=0,\ldots,2m}
\end{equation*}
Observe that each $\lambda_{2i}$ appears in exactly two columns of $\cY_j$. Hence, making a change of variables by interchanging $\lambda_{2i}$ and $\lambda_{2i'}$ for any two $i,i'$ does not change the value of the determinant of~$\cY_j$. Writing $x_i:=\lambda_{2i}$, for $1\leq i\leq m$,  we therefore have
		\begin{equation}\label{7}
\cI_{2m}(u)+\cJ_{2m}(u)= \frac{2c_{2m}}{m!}\sum_{j=0}^{m} \;\binom{m}{j} \;  \int_{\substack{x_1,\ldots, x_j \leq u\\u\leq x_{j+1}, \ldots, x_m}}\det(\cY_j) \; e^{-\sum\limits_{i=1}^{m} \tfrac{x_i^2}{2} } \,  \d x_1 \ldots \d x_{m},
\end{equation}
Using the multilinearity of the determinant we can write $\det(\cY_j)$ as a sum of determinants of matrices, each of which has as double colums either $\left[\begin{smallmatrix} G_{k}(x_i)& P_{k}(x_i)\end{smallmatrix}\right] $ or $\left[\begin{smallmatrix} -G_{k}(u)& P_{k}(x_i)\end{smallmatrix}\right] $. Observe that whenever the double column $\left[\begin{smallmatrix} -G_{k}(u)& P_{k}(x_i)\end{smallmatrix}\right] $ appears twice in a matrix the corresponding determinant equals zero. Moreover, we may interchange the double columns as we wish without changing the value of the determinant. All this yields
	\begin{equation}\label{8}
	\det(\cY_j) = \det(\cK)  - (m-j)\det(\cL),\end{equation}
where the matrices $\cK,\cL\in\HR^{(2m+1)\times (2m+1)}$ are defined as
\begin{align*}\cK&= \begin{bmatrix} P_{k}(u) & \begin{bmatrix}G_{k}(x_i) & P_{k}(x_i)\end{bmatrix}_{i=1,\ldots, m} \end{bmatrix}_{k=0,\ldots,2m},\\
\cL&= \begin{bmatrix}P_{k}(u) &  \begin{bmatrix}G_{k}(x_i) & P_{k}(x_i)\end{bmatrix}_{{i=1,\ldots,m-1 }} &   G_{k}(u) & P_{k}(x_{m})\end{bmatrix}_{k=0,\ldots,2m}.
\end{align*}
Note that $\cK \, e^{-\sum\limits_{i=1}^m\tfrac{x_i^2}{2}}$ is invariant under any permutations of the $x_i$. We may apply \cref{auxiliary_lemma} to conclude that
\begin{align}\label{50}
&\frac{2c_{2m}}{m!}\sum_{j=0}^{m} \binom{m}{j}   \int_{\stackrel{x_1,\ldots, x_j \leq u}{u\leq x_{j+1}, \ldots, x_m}}\det(\cK) \; e^{-\sum\limits_{i=1}^{m} \tfrac{x_i^2}{2} }   \d x_1 \ldots \d x_{m}\\
 \nonumber=\quad& \frac{2c_{2m}}{m!}\int_{x_1,\ldots,x_m\in\HR}\det(\cK) \; e^{-\sum\limits_{i=1}^{m} \tfrac{x_i^2}{2} }   \d x_1 \ldots \d x_{m}\\
\nonumber  =\quad&2c_{2m}\int_{x_1<\ldots<x_m\in\HR}\det(\cK) \; e^{-\sum\limits_{i=1}^{m} \tfrac{x_i^2}{2} }   \d x_1 \ldots \d x_{m}\\
\nonumber  =\quad&2c_{2m}\int_{\lambda_1<\ldots<\lambda_{2m}\in\HR}\det\,  \begin{bmatrix} P_{k}(u) & P_{k}(\lambda_1) & \ldots & P_{k}(\lambda_{2m})\end{bmatrix}_{k=0,\ldots, 2m} \; e^{-\sum\limits_{i=1}^{2m} \tfrac{\lambda_i^2}{2} }   \d \lambda_1 \ldots \d \lambda_{2m}\\
\nonumber    =\quad&2c_{2m}\int_{\lambda_1<\ldots<\lambda_{2m}\in\HR}\big[\prod_{i=1}^{2m} (\lambda_i-u)\big] \Delta(\lambda) \; e^{-\sum\limits_{i=1}^{2m} \tfrac{\lambda_i^2}{2} }   \d \lambda_1 \ldots \d \lambda_{n}\\
   =\quad &2\cJ_n(u),\nonumber
\end{align}
the fifth line line by \cref{b} and the last line by \cref{J_n}. Combining this with \cref{7} and \cref{8} we get
	\begin{align*}
	\cI_{2m}(u)&=(\cI_{2m}(u)+\cJ_{2m}(u)) - \cJ_{2m}(u)\\
	 &=\cJ_{2m}(u)-\frac{2c_{2m}}{m!}\sum_{j=0}^{m} \;\binom{m}{j} \;  \int_{\stackrel{x_1,\ldots, x_j \leq u}{u\leq x_{j+1}, \ldots, x_m}}(m-j)\det(\cL) \; e^{-\sum\limits_{i=1}^{m} \tfrac{x_i^2}{2} } \,  \d x_1 \ldots \d x_{m}\\
	 &=\cJ_{2m}(u)-\frac{2c_{2m}}{(m-1)!}\sum_{j=0}^{m-1} \;\binom{m-1}{j} \;  \int_{\stackrel{x_1,\ldots, x_j \leq u}{u\leq x_{j+1}, \ldots, x_m}}\det(\cL) \; e^{-\sum\limits_{i=1}^{m} \tfrac{x_i^2}{2} } \,  \d x_1 \ldots \d x_{m}
	\end{align*}
The function $\det(\cL)\,e^{-\sum\limits_{i=1}^{m-1} \tfrac{x_i^2}{2}}$ is invariant under permuting $x_1,\ldots,x_{m-1}$ (excluding $x_m$!). Therefore, we can apply \cref{auxiliary_lemma} to obtain
	\begin{equation*}
	\cI_{2m}(u)=\cJ_{2m}(u)-\frac{2c_{2m}}{(m-1)!} \,\int_{x_1,\ldots,x_{m-1} \in\HR} \left[\;\int_{x_m=u}^\infty \det(\cL) e^{-\tfrac{x_m^2}{2} } \d x_m \right] e^{-\sum\limits_{i=1}^{m-1} \tfrac{x_i^2}{2} } \,  \d x_1 \ldots \d x_{m-1}.
\end{equation*}
Furthermore, $x_m$ appears in one single column in $\cL$. Integrating over $x_m$ in therefore shows that
	\begin{align*}
	&\int_{x_m=u}^\infty \det(\cL) e^{-\tfrac{x_m^2}{2} } \d x_m\\
	=& \det \begin{bmatrix}P_{k}(u) &&  \left[G_{k}(x_i)\, P_{k}(x_i)\right]_{{i=1,\ldots,m-1 }} &&   G_{k}(u)& & G_k(\infty)-G_k(u)\end{bmatrix}_{k=0,\ldots,2m}\\
	=&\det \begin{bmatrix}P_{k}(u) &&  \left[G_{k}(x_i)\, P_{k}(x_i)\right]_{{i=1,\ldots,m-1 }} &&   G_{k}(u)& & G_k(\infty)\end{bmatrix}_{k=0,\ldots,2m}.
	\end{align*}
Let us denote by $\cM$ the last matrix:
$$\cM:=\begin{bmatrix}P_{k}(u) &&  \left[G_{k}(x_i)\, P_{k}(x_i)\right]_{{i=1,\ldots,m-1 }} &&   G_{k}(u)& & G_k(\infty)\end{bmatrix}_{k=0,\ldots,2m} \in \HR^{(2m+1)\times (2m+1)}.$$
From \cref{prop1} we get
\begin{align*}
	&\int_{x_1,\ldots,x_{m-1} \in\HR} \det(\cM) e^{-\sum\limits_{i=1}^{m-1} \tfrac{x_i^2}{2} } \,  \d x_1 \ldots \d x_{m-1}\\
	= \;&\sqrt{2\pi} (m-1)! \,2^{\,m-1}\,e^{-\tfrac{u^2}{2}}\,\sum_{1\leq i,j \leq m} \det(\Gamma_1^{i,j}) \;\det\begin{bmatrix}  P_{2j}(u)&P_{2i-1}(u)  \\ P_{2j-1}(u)&P_{2i-2}(u)   \end{bmatrix}.
	\end{align*}
where the matrix $\Gamma_1^{i,j}$ is defined as in \cref{gamma_matrices}.
Altogether, we therefore have
\begin{equation*}
	\cI_{2m}(u)=\cJ_{2m}(u)-c_{2m}\sqrt{2\pi} \,2^{\,m}\,e^{-\tfrac{u^2}{2}}\,\sum_{1\leq i,j \leq m} \det(\Gamma_1^{i,j}) \; \det\begin{bmatrix}  P_{2j}(u)&P_{2i-1}(u)  \\ P_{2j-1}(u)&P_{2i-2}(u)   \end{bmatrix}
	\end{equation*}
Finally, we substitute $c_{2m} =\big(2^m \prod_{i=1}^{2m}\Gamma\left(\tfrac{i}{2}\right)\big)^{-1}$ and put the minus into the determinant to obtain
\begin{equation*}
	\cI_{2m}(u)=\cJ_{2m}(u)+\frac{\sqrt{2\pi} \,e^{-\tfrac{u^2}{2}}}{\prod_{i=1}^{2m}\Gamma\left(\tfrac{i}{2}\right)}\,\sum_{1\leq i,j \leq m} \det(\Gamma_1^{i,j}) \; \det\begin{bmatrix} P_{2i-1}(u) & P_{2j}(u) \\ P_{2i-2}(u) & P_{2j-1}(u) \end{bmatrix}
	\end{equation*}
This finishes the proof.\qed
%%%%%%%%%%%%%%%%%%%%%%%%%%%%%%%%%%%%%%%%%%%%%%%%%%%%%%%%%%%%%%%%%%%%%%%
%%%%%%%%%%%%%%%%%%%%%%%%%%%%%%%%%%%%%%%%%%%%%%%%%%%%%%%%%%%%%%%%%%%%%%%
%%%%%%%%%%%%%%%%%%%%%%%%%%%%%%%%%%%%%%%%%%%%%%%%%%%%%%%%%%%%%%%%%%%%%%%
\subsection{The case when $n=2m-1$ is odd} \label{sec_n_odd}
 We proceed as in the preceeding section and can therefore be brief in our explanations. In \cref{a} we integrate over all the $\lambda_i$ with $i$ odd to obtain
	\begin{equation}\label{602}
\cI_{2m-1}(u)+\cJ_{2m-1}(u)= 2c_{2m-1}\sum_{j=0}^{m-1} \; \int_{\substack{ x_1\leq x_2\leq ...\leq x_{j} \leq u\\u\leq x_{j+1}\leq \ldots \leq x_{m-1}}}\det(\cX_j) \; e^{-\sum\limits_{i=1}^{m}\tfrac{x_i^2}{2}} \, \d x_1 \ldots \d x_{m-1},
\end{equation}
where $x_i:=y_{2i}$, $1\leq i\leq m-1$ and $\cX_j$ is the matrix
\begin{align*}
\cX_j=&\Big[ P_{k}(u) \hspace{0.6cm} \begin{bmatrix}G_{k}(x_i) - G_k(x_{i-1}) & P_{k}(x_i)\end{bmatrix}_{{i=1,\ldots j}}\;  \begin{bmatrix}G_{k}(x_{j+1}) - G_k(u) & P_{k}(x_{j+1})\end{bmatrix}\,\ldots  \\
&\hspace{1cm} \ldots\, \begin{bmatrix}G_{k}(x_i) - G_{k}(x_{i-1}) & P_{k}(x_i)\end{bmatrix}_{{i=j+2,\ldots, m-1}} \hspace{0.6cm} G_k(\infty)-G_k(x_{m-1}) \Big]_{k=0,\ldots,{2m-1}}.
\end{align*}
We have $\det(\cX_j)=\det(\cY_j)$, where
\begin{align*}
\cY_j=&\Big[ P_{k}(u) \hspace{0.6cm} \begin{bmatrix}G_{k}(x_i) & P_{k}(x_i)\end{bmatrix}_{{i=1,\ldots j}}\; \ldots  \\
&\hspace{1cm} \ldots\, \begin{bmatrix}G_{k}(x_i) - G_{k}(u) & P_{k}(x_i)\end{bmatrix}_{{i=j+1,\ldots, m-1}} & G_k(\infty)-G_k(u)\Big]_{k=0,\ldots,{2m-1}}.
\end{align*}
Permuting $x_{1},\ldots, x_{j}$ or permuting $x_{j+1},\ldots,x_{m}$ does not change the value of $\det(\cY_j)$. Hence, $\cI_{2m-1}(u)+\cJ_{2m-1}(u)$ is equal to
		\begin{equation*}
\frac{2c_{2m-1}}{(m-1)!}\sum_{j=0}^{m-1} \;\binom{m-1}{j}\; \int_{\substack{x_1,\ldots, x_j \leq u\\u\leq x_{j+1}, \ldots, x_{m-1}}}\det(\cY_j) \; e^{-\sum\limits_{i=1}^{m-1}  \tfrac{x_i^2}{2}} \, \d x_1 \ldots \d x_{m-1},
\end{equation*}
Using the multilinearity of the determinant we have
	\[\det(\cY_j) = \det(\cK) -  \det(\cM)- (m-1-j)\det(\cL),\]
where $\cK,\cM,\cL\in\HR^{2m\times 2m}$ are the matrices defined by
\begin{align*}
\cK&= \begin{bmatrix} P_{k}(u) & \begin{bmatrix}G_{k}(x_i) & P_{k}(x_i)\end{bmatrix}_{i=1,\ldots, m-1} & G_k(\infty)\end{bmatrix}_{k=0,\ldots,2m-1},\\
\cM&= \begin{bmatrix} P_{k}(u) & \begin{bmatrix}G_{k}(x_i) & P_{k}(x_i)\end{bmatrix}_{i=1,\ldots, m-1} & G_k(u)\end{bmatrix}_{k=0,\ldots,2m-1}, \\
\cL&=\begin{bmatrix}P_{k}(u) &  \begin{bmatrix}G_{k}(x_i) & P_{k}(x_i)\end{bmatrix}_{{i=1,\ldots,m-2}} &   G_{k}(u) & P_{k}(x_{m-1})&  G_k(\infty)-G_k(u)\end{bmatrix}_{k=0,\ldots,2m-1}.
\end{align*}
Integrating $\int_{x_{m-1}>u } \det(\cL) \,e^{-\tfrac{x_{m-1}^2}{2}}\;\d x$ replaces the $P_k(x_{m-1})$ in $\cL$ by $G_k(\infty)-G_k(u)$. Hence, this integral is equal to
\begin{align*}
\det \begin{bmatrix}P_{k}(u) &  \hspace{-0.2cm}\begin{bmatrix}G_{k}(x_i) & P_{k}(x_i)\end{bmatrix}_{{i=1,\ldots,m-2 }} &  \hspace{-0.2cm} G_{k}(u) & G_k(\infty)-G_k(u)&  G_k(\infty)-G_k(u)\end{bmatrix}_{k=0,\ldots,2m-1},
\end{align*}
which is equal to zero, because it has two columns which are equal. Therefore, we have
	\begin{align*}
	& \sum_{j=0}^{m-1}\binom{m-1}{j}\; \int_{\substack{x_1,\ldots, x_j \leq u\\u\leq x_{j+1}, \ldots, x_{m-1}}}(m-1-j)\det(\cL) \; e^{-\sum\limits_{i=1}^{m-1}  \tfrac{x_i^2}{2}}  \d x_1 \ldots \d x_{m-1}  \\
	 =&\sum_{j=0}^{m-1}\binom{m-1}{j}\; \int_{\substack{x_1,\ldots, x_j \leq u\\u\leq x_{j+1}, \ldots, x_{m-2}}} \left[\int_{x_{m-1}=u}^\infty  \det(\cL) \; e^{-\tfrac{x_{m-1}^2}{2}} \d x_{m-1}\right] e^{-\sum\limits_{i=1}^{m-2}  \tfrac{x_i^2}{2}} \d x_1 \ldots \d x_{m-2}\\
	  = &\;0.
	\end{align*}
Altogether, we have that $\cI_{2m-1}(u)+\cJ_{2m-1}(u)$ is equal to
	\begin{align*}
	\frac{2c_{2m-1}}{(m-1)!}\sum_{i=0}^{m-1} \;\binom{m-1}{j}\; \int_{\substack{x_1,\ldots, x_j \leq u\\u\leq x_{j+1}, \ldots, x_{m-1}}}(\det(\cK)+\det(\cM)) \; e^{-\sum\limits_{i=1}^{m-1} \tfrac{x_i^2}{2}}  \d x_1 \ldots \d x_{m-1}
\end{align*}
By construction, both $\det(\cK)$ and $\det(\cM)$ are invariant under any permutation of the $x_i$. We may apply \cref{auxiliary_lemma} to get
\begin{equation*}
\cI_{2m-1}(u)+\cJ_{2m-1}(u) = \frac{2c_{2m-1}}{(m-1)!} \int_{x_1,\ldots,x_{m-1}\in\HR}(\det(\cK)-\det(\cM)) \; e^{-\sum\limits_{i=1}^{m-1} \tfrac{x_i^2}{2}}  \d x_1 \ldots \d x_{m-1}
	\end{equation*}
Similar to \cref{50} we deduce that
\begin{equation*}
\frac{2c_{2m-1}}{(m-1)!} \int_{x_1,\ldots,x_{m-1}\in\HR}\det(\cK) \; e^{-\sum\limits_{i=1}^{m-1} \tfrac{x_i^2}{2}}  \d x_1 \ldots \d x_{m-1} = 2\cJ_{2m-1}(u),
\end{equation*}
so that
\begin{align}\label{102}
	\cI_{2m-1}(u)&=(\cI_{2m-1}(u)+\cJ_{2m-1}(u))-\cJ_{2m-1}(u)\\
	&= \cJ_{2m-1}(u)-\frac{2c_{2m-1}}{(m-1)!} \int_{x_1,\ldots,x_{m-1}\in\HR} \det(\cM)\; e^{-\sum\limits_{i=1}^{m-1} \tfrac{x_i^2}{2}}  \d x_1 \ldots \d x_{m-1}\nonumber
\end{align}
By \cref{prop2} we have
	\begin{align*}
	&\int_{x_1,\ldots,x_{m-1}\in\HR} \det(\cM) \, e^{-\sum\limits_{i=1}^{m-1} \tfrac{x_i^2}{2}}\, \d x_1 \ldots \d x_{m-1}\\
	= &(m-1)! \,2^{m-1}\,e^{-\tfrac{u^2}{2}}\,\sum_{0\leq i,j\leq m-1}  \det(\Gamma_2^{i,j})\, \det\begin{bmatrix}  P_{2i+1}(u) & P_{2j}(u)\\    P_{2i}(u) & P_{2j-1}(u) \end{bmatrix},
	\end{align*}
where the matrix $\Gamma_2^{i,j}$ is defined as in \cref{gamma_matrices}.
Combining this equation with \cref{102}, substituting $c_{2m-1} =\sqrt{2}\,\big(2^m \prod_{i=1}^{2m-1}\Gamma\left(\tfrac{i}{2}\right)\big)^{-1}$ and putting the minus into the determinant we get
	\[\cI_{2m-1}(u)=\cJ_{2m-1}(u)  +  \frac{\sqrt{2}\,e^{-\tfrac{u^2}{2}}}{\prod_{i=1}^{2m-1}\Gamma\left(\tfrac{i}{2}\right)}\,\sum_{0\leq i,j\leq m-1}  \det(\Gamma_2^{i,j})\, \det\begin{bmatrix} P_{2j}(u) &   P_{2i+1}(u)\\   P_{2j-1}(u) & P_{2i}(u)\end{bmatrix}.\]
This finishes the proof.\qed
%%%%%%%%%%%%%%%%%%%%%%%%%%%%%%%%%%%%%%%%%%%%%%%%%%%%%%%%%%%%%%%%%%%%%%%
%%%%%%%%%%%%%%%%%%%%%%%%%%%%%%%%%%%%%%%%%%%%%%%%%%%%%%%%%%%%%%%%%%%%%%%
%%%%%%%%%%%%%%%%%%%%%%%%%%%%%%%%%%%%%%%%%%%%%%%%%%%%%%%%%%%%%%%%%%%%%%%
%%%%%%%%%%%%%%%%%%%%%%%%%%%%%%%%%%%%%%%%%%%%%%%%%%%%%%%%%%%%%%%%%%%%%%%
%%%%%%%%%%%%%%%%%%%%%%%%%%%%%%%%%%%%%%%%%%%%%%%%%%%%%%%%%%%%%%%%%%%%%%%
%%%%%%%%%%%%%%%%%%%%%%%%%%%%%%%%%%%%%%%%%%%%%%%%%%%%%%%%%%%%%%%%%%%%%%%
%%%%%%%%%%%%%%%%%%%%%%%%%%%%%%%%%%%%%%%%%%%%%%%%%%%%%%%%%%%%%%%%%%%%%%%
%%%%%%%%%%%%%%%%%%%%%%%%%%%%%%%%%%%%%%%%%%%%%%%%%%%%%%%%%%%%%%%%%%%%%%%
%%%%%%%%%%%%%%%%%%%%%%%%%%%%%%%%%%%%%%%%%%%%%%%%%%%%%%%%%%%%%%%%%%%%%%%
%%%%%%%%%%%%%%%%%%%%%%%%%%%%%%%%%%%%%%%%%%%%%%%%%%%%%%%%%%%%%%%%%%%%%%%
%%%%%%%%%%%%%%%%%%%%%%%%%%%%%%%%%%%%%%%%%%%%%%%%%%%%%%%%%%%%%%%%%%%%%%%
%%%%%%%%%%%%%%%%%%%%%%%%%%%%%%%%%%%%%%%%%%%%%%%%%%%%%%%%%%%%%%%%%%%%%%%
%%%%%%%%%%%%%%%%%%%%%%%%%%%%%%%%%%%%%%%%%%%%%%%%%%%%%%%%%%%%%%%%%%%%%%%
%%%%%%%%%%%%%%%%%%%%%%%%%%%%%%%%%%%%%%%%%%%%%%%%%%%%%%%%%%%%%%%%%%%%%%%
{

}
%%%%%%%%%%%%%%%%%%%%%%%%%%%%%%%%%%%%%%%%%%%%%%%%%%%%%%%%%%%%%%%%%%%%%%%
%%%%%%%%%%%%%%%%%%%%%%%%%%%%%%%%%%%%%%%%%%%%%%%%%%%%%%%%%%%%%%%%%%%%%%%
%%%%%%%%%%%%%%%%%%%%%%%%%%%%%%%%%%%%%%%%%%%%%%%%%%%%%%%%%%%%%%%%%%%%%%%
%%%%%%%%%%%%%%%%%%%%%%%%%%%%%%%%%%%%%%%%%%%%%%%%%%%%%%%%%%%%%%%%%%%%%%%
\newpage
\appendix
\section{Sage code to compute $E(n,p)$}\label{appendix}
Below is the code for two \textsc{sage} scripts, that compute $E(2m+1,p)$ and $E(2m,p)$, respectively.

{\scriptsize
\begin{lstlisting}
#####################################################################
# The case n = 2m+1 is odd: Compute A and A_1, so that E(n,p)=1+A*A_1
# (A is the factor in front of the sum and A_1 is the sum)
#####################################################################
######### define the variables
p,m,n,i,j,x=var('p,m,n,i,j,x');
######### Set the value of m
m=2; n=2*m+1;
######### compute A
A(p)=sqrt(pi)*sqrt(p-1)^(n-2)*sqrt(3*p-2)/(prod(gamma(i/2) for i in (1..n)));
######### compute the determinants of the Gamma matrices
G_1 = matrix(m, lambda i,j: gamma(i+j+3/2));
G_1 = matrix(m, lambda i,j:
	det(G_1.matrix_from_rows_and_columns([0..i-1,i+1..m-1],[0..j-1,j+1..m-1])));
######### compute A_1
from sage.misc.mrange import cantor_product
L = list(cartesian_product_iterator([[1..m], [1..m]]));
A_1(p) = sum(G_1[i-1,j-1] * gamma(i+j-1/2) * ((1-2*i+2*j)/(3-2*i-2*j))
	* (-(3*p-2)/(4*(p-1)))^(1-i-j)
	* hypergeometric([2-2*i,1-2*j],[5/2-i-j],(3*p-2)/(4*(p-1))) for (i,j) in L);
A_1(p) = A_1(p).simplify_hypergeometric();
######### compute E(n,p)
E_n_odd(p)=A(p)*A_1(p);
E_n_odd(p)=E_n_odd(p).factor()
E_n_odd(p)=E_n_odd(p)+1
print(E_n_odd(p)) # prints the formula (wrap 'latex()' around it to get tex code)
\end{lstlisting}

\begin{lstlisting}
####################################################################################
# The case n = 2m is even: Compute B, B_1, B_2 and B_3, so that E(n,p)=B*(B_1-B_2+B_3)
# (B is the factor in front of the sum and B_1-B_2+B_3 is the sum
####################################################################################
######### define the variables
p,m,n,i,j,x=var('p,m,n,i,j,x');
######### Set the value of m
m=2; n=2*m;
######### compute B
B(p)=sqrt(p-1)^(n-2)*sqrt(3*p-2)/(prod(gamma(i/2) for i in (1..n)));
######### compute the determinants of the Gamma matrices
G_2 = matrix(m, lambda i,j: gamma(i+j+1/2));
G_2 = matrix(m, lambda i,j:
	det(G_2.matrix_from_rows_and_columns([0..i-1,i+1..m-1],[0..j-1,j+1..m-1])));
######### compute B_1
B_1(p) = sum(sqrt(pi) * G_2[0,j] * (gamma(2*j+2)/((-1)^j*4^j*gamma(j+1)))
	* ((p-2)^j*p)/((p-1)^j*(3*p-2))
	* hypergeometric([-j, 1/2],[3/2], -p^2/((3*p-2)*(p-2))) for j in [0..m-1]);
B_1(p) = B_1(p).simplify_hypergeometric();
######### compute B_2
B_2(p) = sum(G_2[0,j] * gamma(j+1/2) * (-4*(p-1)/(3*p-2))^(j+1) /2 for j in [0..m-1]);
######### compute B_3
from sage.misc.mrange import cantor_product
L = list(cartesian_product_iterator([[1..m-1], [0..m-1]]));
B_3(p) = sum(G_2[i,j] * gamma(i+j+1/2) * ((1-2*i+2*j)/(1-2*i-2*j))
	* (-4*(p-1)/(3*p-2))^(i+j)
	* hypergeometric([-2*j,-2*i+1],[3/2-i-j],(3*p-2)/(4*(p-1))) for (i,j) in L);
B_3(p) = B_3(p).simplify_hypergeometric();
######### compute E(n,p)
E_n_even(p)=B(p) * (B_1(p)-B_2(p)+B_3(p));
E_n_even(p)=E_n_even(p).factor()
print(E_n_even(p)) # prints the formula (wrap 'latex()' around it to get tex code)
\end{lstlisting}
}

\end{document}